\documentclass[11pt]{amsart}
\usepackage{latexsym}
\usepackage{amssymb, amsmath,mathrsfs, mathtools}
\usepackage[left=2.5cm,right=2.5cm,top=2.5cm, bottom=2.5cm ]{geometry}
\usepackage[OT2,T1]{fontenc}
\usepackage{todonotes}
\usepackage{setspace}
\usepackage{graphicx, subfigure}
\usepackage{caption}
\usepackage{hyperref}

\makeatletter
\@namedef{subjclassname@2020}{%
  \textup{2020} Mathematics Subject Classification}
\makeatother

\newtheorem{theorem}{Theorem}[section]

\newtheorem{conjecture}[theorem]{Conjecture}

\newtheorem{corollary}[theorem]{Corollary}

\theoremstyle{remark}
\newtheorem{remark}[theorem]{Remark}

\DeclareSymbolFont{cyrletters}{OT2}{wncyr}{m}{n}
\DeclareMathSymbol{\Sha}{\mathalpha}{cyrletters}{"58}
\author[Bickel]{Kelly Bickel}
\address{Department of Mathematics, Bucknell University, Olin Science Building, Lewisburg, PA 17837, USA.}
\email{kelly.bickel@bucknell.edu}


\author[Corbett]{Georgia Corbett}
\address{Department of Mathematics, Bucknell University, Olin Science Building, Lewisburg, PA 17837, USA.}
\email{gsfc001@bucknell.edu}

\author[Glenning]{Annie Glenning}
\address{Department of Mathematics, Bucknell University, Olin Science Building, Lewisburg, PA 17837, USA.}
\email{asg019@bucknell.edu}

\author[Guan]{Changkun Guan}
\address{Department of Mathematics, Bucknell University, Olin Science Building, Lewisburg, PA 17837, USA.}
\email{cg034@bucknell.edu}

\author[Vollmayr-Lee]{Martin Vollmayr-Lee}
\address{Department of Mathematics and Statistics, Haverford College, Haverford, PA 19041, USA.}
\email{mvollmayrl@haverford.edu}

\begin{document}

\title{Crouzeix's Conjecture, compressions of shifts, and classes of nilpotent matrices}
\date{\today}

\maketitle
\begin{abstract} This paper studies the level set Crouzeix conjecture, which is a weak version of Crouzeix's conjecture that applies to finite compressions of the shift. Amongst other results, this paper establishes the level set Crouzeix conjecture for several classes of $3\times3$, $4\times4$, and $5\times5$ matrices associated to compressions of the shift via a geometrix analysis of their numberical ranges. This paper also establishes Crouzeix's conjecture for several classes of nilpotent matrices whose studies are motivated by related compressions of shifts.

\end{abstract}
\section{Introduction} 
\subsection{General Motivation and Background}

Let $A$ be an $n\times n$ matrix and let $\|\cdot \|$ denote the length of a vector in $\mathbb{C}^n$. Then the numerical range of $A$ is the set of numbers in $\mathbb{C}$ given by
\[ W(A) = \left\{ \left \langle A x, x \right \rangle : x\in \mathbb{C}^n \text{ with } \| x \| =1\right\}.\]
In this matrix setting, the numerical range is a compact, convex set that includes the spectrum of $A$.  It can be used to approximate the eigenvalues of $A$, but also encodes more information about $A$ than the eigenvalues alone do; for example, a matrix $A$ is self-adjoint if and only if $W(A) \subseteq \mathbb{R}$. 

Our interest in the numerical range stems partially from an important open problem in operator theory called Crouzeix's conjecture, which posits that for any polynomial $p$, $W(A)$ can be used to obtain a good bound on the operator norm of the matrix $p(A)$, denoted $\| p(A)\|$. Specifically,  in \cite{Cr07} in 2007, Crouzeix stated his now-famous conjecture: for all polynomials $p$ and square matrices $A$, 
\begin{equation} \label{eqn:cc1} \|p(A) \| \le 2 \max_{z\in W(A)} |p(z)|.\end{equation}
In earlier work \cite{Crouzeix1}, Crouzeix established \eqref{eqn:cc1} for $2\times 2$ matrices and in \cite{Cr07}, Crouzeix established the full inequality \eqref{eqn:cc1} with $11.08$ in place of $2$. The best known general result is due to Crouzeix and Palencia, who showed that \eqref{eqn:cc1} holds when $2$ is replaced with $1+\sqrt{2}\approx 2.41$ in \cite{CP17}. Since 2007, Crouzeix's conjecture has been proven for a number of specific classes of matrices; these include perturbed Jordan blocks and related matrices \cite{GreenbaumChoi, CH13}, nilpotent $3\times 3$ matrices \cite{C13}, and $3\times 3$ tridiagonal Toeplitz matrices \cite{GKL}. For additional recent results related to Crouzeix's conjecture, see \cite{CGL18, CG, COR, Ov}. 

This paper is motivated by recent work by the first author, P. Gorkin, and other collaborators in \cite{BGGRSW, BG21} on connections between Crouzeix's conjecture and a specific class of matrices called $S_n$ matrices. Specifically, an $n \times n$ matrix $A$ is of class $S_n$ if $A$ is a contraction (i.e. $\| A \| \le 1$), the eigenvalues of $A$ are in $\mathbb{D}$, and $A$ has defect index one (i.e. $\mbox{rank} (I - A^\star A) =1$), see \cite{GauWu}. For example, the $n\times n$ perturbed Jordan blocks studied by Choi and Greenbaum in \cite{GreenbaumChoi} were of the form
\[J^a_n =  \begin{pmatrix} 0 & 1 & \\ & \ddots & \ddots \\ & & \ddots & 1\\ a & & & 0 \end{pmatrix}\]
and if $|a|<1$, then these are of class $S_n$. Since Choi and Greenbaum established Crouzeix's conjecture for these perturbed Jordan blocks, it makes sense to ask whether Crouzeix's conjecture might be particularly tractable for matrices of class $S_n.$

It furthermore turns out that every matrix $A$ of class $S_n$ is unitarily equivalent to an upper triangular $n\times n$ matrix $M$ with entries given by
\begin{equation} \label{eqn:M} M_{ij} = \left \{ \begin{array}{cc} a_i & \text{ if } i = j \\
0  & \text{ if } i > j \\
\sqrt{1-|a_i|^2} \sqrt{1-|a_j|^2} \prod_{k=i+1}^{j-1}(-\bar{a}_k) 
& \text{ if } i < j
\end{array}\right.,\end{equation}
where $a_1, \dots, a_n$ in the unit disk $\mathbb{D}$ are the eigenvalues of $A$. Clearly, $M$ is uniquely determined by the numbers $a_1, \dots, a_n.$ Such a list of numbers also gives rise to a natural rational function $\Theta$, called a finite Blaschke product, using the following formula
\begin{equation} \label{eqn:theta} \Theta(z) = \lambda \prod_{j=1}^n \frac{z-a_j}{1-\bar{a}_j z},\end{equation}
where $\lambda$ can be chosen to be any number in the unit circle $\mathbb{T},$ though we will generally assume that $\lambda=1.$ Finite Blaschke products are holomorphic on $\mathbb{D}$ and satisfy $|\Theta|=1$ on $\mathbb{T}$. They also play a crucial role in classical complex analysis; for example, they appear naturally in interpolation, zero-factorization, and approximation problems; indeed, they basically act as polynomials in settings where the underlying domain is $\mathbb{D}$ instead of $\mathbb{C},$ see the books \cite{GMR, Garnett} for details.

Conversely, one could start with a finite Blaschke product $\Theta$ as in \eqref{eqn:theta}, extract its zeros $a_1, \dots, a_n$ and use those to obtain a matrix $M$ as in \eqref{eqn:M}. To emphasize that we typically use this order of operations, we will denote the resulting matrix in \eqref{eqn:M} by $M_\Theta.$ This correspondence is deeper than initially apparent. 
To see this, let $H^2(\mathbb{D})$ denote the standard Hardy space on the unit disk and let $S$ denote the shift operator given by $(Sf)(z) = zf(z)$. The space $\Theta H^2(\mathbb{D})$ is a closed subspace of $H^2(\mathbb{D})$ and $K_\Theta := H^2(\mathbb{D}) \ominus  \Theta H^2(\mathbb{D})$ is called the associated model space. Then $K_\Theta$ has dimension $n$ and it is natural to consider the compression of the shift $S$ to $K_\Theta$, denoted $S_\Theta$, and defined as follows
\[ S_\Theta := P_\Theta S|_{K_\Theta},\]
 where $P_\Theta$ is the orthogonal projection from $H^2(\mathbb{D})$ onto $K_\Theta$. The matrix $M_\Theta$ given in \eqref{eqn:M} is actually a matrix representation of  $S_\Theta$ with respect to a particularly nice orthonormal basis. Since we do not need all of the details here, we refer the interested reader to the book \cite{DGSV} and the references therein for more information.

\subsection{Prior Key Results} Studying Crouzeix's conjecture for matrices of the form \eqref{eqn:M} is equivalent to studying Crouzeix's conjecture for the entire class of $S_n$ matrices, which is a key goal of the current paper. 
 As our investigations are closely motivated by some of the ideas and results from \cite{BG21}, we need to describe those here. In \cite{BG21}, the first author and P. Gorkin observed that if Crouzeix's conjecture holds for some matrix $M_\Theta$, then for every finite Blaschke product $B$ with $\deg B < \deg \Theta$, it follows that
\begin{equation} \label{eqn:LSCC} \max_{z \in W(M_\Theta)} |B(z)| \ge \tfrac{1}{2}. \end{equation} 
This has a natural geometric interpretation. Specifically, define the $1/2$-level set of $B$ by 
\[ \Omega^B_{1/2} =\left \{ z\in \mathbb{C}: |B(z)| < \tfrac{1}{2} \right\}.\]
Then the statement that \eqref{eqn:LSCC} should hold can be rephrased as the following, which we call the level set Crouzeix conjecture (or LSC conjecture for short):
\begin{conjecture}\textbf{(Level set Crouzeix conjecture)} If $\Theta$ and $B$ are finite Blaschke products with $\deg B < \deg \Theta$,  then $\displaystyle W(M_\Theta) \not \subseteq  \Omega^B_{1/2}.$
\end{conjecture}
This conjecture gives us a way to visualize the success (or failure) of Crouzeix's conjecture by graphing sets in the complex plane that are closely connected to simple holomorphic functions. As evidenced in Figure \ref{fig:LSCC} where we draw boundaries of the key sets, the non-containment statement $W(M_\Theta) \not \subseteq \Omega_{1/2}^B$ posited in the LSC conjecture appears quite nontrivial. While the paper \cite{BG21} established the LSC conjecture for certain classes of  pairs of $B$ and $\Theta$,  most cases remain open. 

\begin{figure}[h!] 
    \subfigure[$B$ with zeros $5/9+i/5$ and $2/3-2i/9$]
      {\includegraphics[width=0.3 \textwidth]{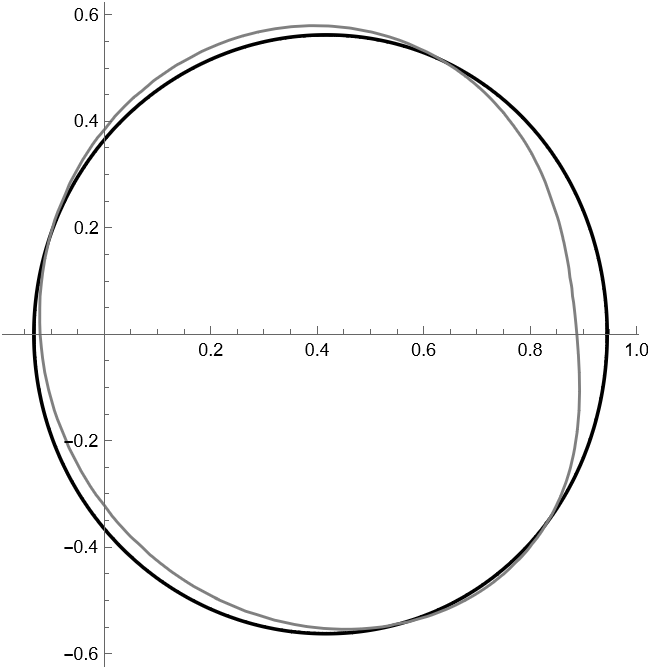}}
    \quad 
    \subfigure[$B$ with zeros $5/9$ and $3/4$]
      {\includegraphics[width=0.3 \textwidth]{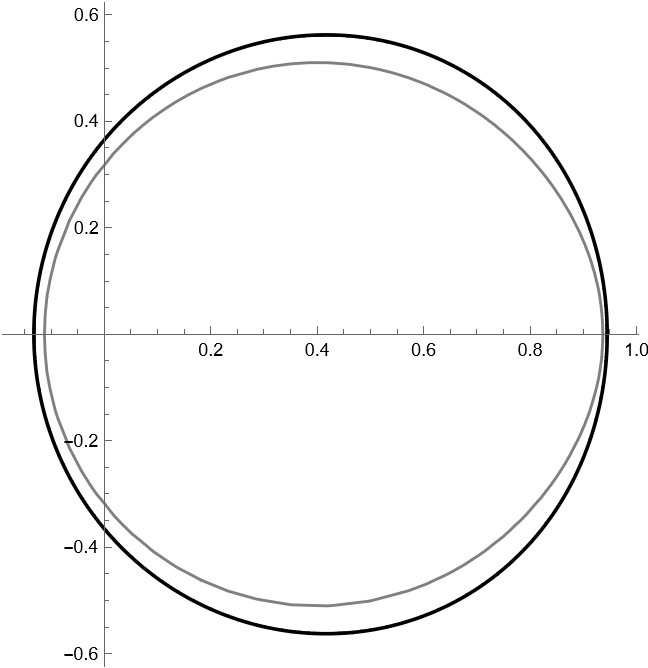}}
      \quad
          \subfigure[$B$ with zeros $1/2-2i/11$ and $2/3-i/4$]
      {\includegraphics[width=0.3 \textwidth]{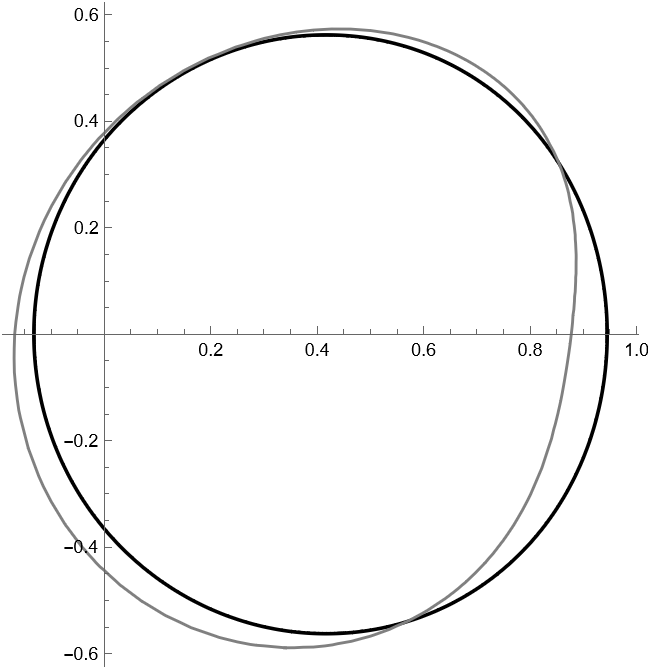}}
  \caption{\textsl{$\partial W(M_\Theta)$  for $\Theta(z)=\Big(\frac{1/2-z}{1-z/2}\Big)^3$ (in black) and  $\partial \Omega^B_{1/2}$ for several $B$ with $\deg B=2$ (in gray).}}
  \label{fig:LSCC}
\end{figure}

One particularly useful result gives a sufficient (though not necessary) condition for the LSC conjecture to hold. To state it, we need notation for disks in $\mathbb{D}.$  Specifically, let $D(c, R)$ denote a Euclidean disk in $\mathbb{D}$ with center $c$ and radius $R$ and $D_\rho(z_0, r)$ denote the pseudohyperbolic disk with (pseudohyperbolic) center $z_0 \in \mathbb{D}$ and (pseudohyperbolic) radius $r \in (0,1)$. Then
 \begin{equation} \label{eqn:diskform} D(c,R) =  \left \{ z\in \mathbb{D}:|z-c|<R \right\} \ \ \text{ and } \ \ D_\rho(z_0,r) = \left \{ z\in \mathbb{D}: \left | \frac{z-z_0}{1-\bar{z}_0 z} \right | <r \right\}.\end{equation} 
 Every pseudohyperbolic disk is actually a Euclidean disk in $\mathbb{D}$ and every Euclidean disk in $\mathbb{D}$ is a pseudohyperbolic disk. 
It turns out that if $W(M_\Theta)$ contains a sufficiently large pseudohyperbolic disk (where ``large'' is measured in terms of the pseudohyperbolic radius), then $\Theta$ satisfies the LSC conjecture (for any $B$). This appears as Corollary 3.3 in \cite{BG21} and the details are below.

\begin{theorem}\textbf{(Pseudohyperbolic disk criterion}) \label{thm:phdcP} Let $\Theta$ be a finite Blaschke product with $\deg \Theta=n$. If there is a $z_0 \in \mathbb{D}$ such that the pseudohyperbolic disk $D_\rho(z_0, (\tfrac{1}{2})^{1/(n-1)}) \subseteq W(M_\Theta)$, then $\displaystyle W(M_\Theta) \not \subseteq  \Omega^B_{1/2}$
for all finite Blaschke products $B$ with $\deg B <n.$
\end{theorem}

We say $\Theta$ \textbf{satisfies the pseudohyperbolic disk criterion} if there is a $z_0 \in \mathbb{D}$ such that the pseudohyperbolic disk $D_\rho(z_0, (\tfrac{1}{2})^{1/(n-1)}) \subseteq W(M_\Theta)$.
The paper \cite{BG21} applies this criterion in the setting of functions $\Theta$ with a single repeated zero. Specifically, for $t\in [0,1)$ and a positive integer $n$, define the finite Blaschke product $\Theta_t$ by 
\begin{equation}\label{eqn:unicritical}
     \Theta_t (z): = \left( \frac{z-t}{1-tz}\right)^n.
\end{equation}

When $n=3$ and $n=4$, Theorem 5.3 in \cite{BG21} established the pseudohyperbolic disk criterion for such $\Theta_t$. Here are the details.

\begin{theorem} \label{thm:pshds} Fix $t\in[0,1)$. If $\deg \Theta_t =3$, then $W(M_{\Theta_t})$ contains a pseudohyperbolic disk with radius $\tfrac{1}{2^{1/2}}$ and if $\deg \Theta_t =4$, then $W(M_{\Theta_t})$ contains a pseudohyperbolic disk with radius $\tfrac{1}{2^{1/3}}$.
\end{theorem}
Part of the proof requires noting that the  $M_{\Theta_t}$ matrices are closely connected to a class of nilpotent matrices. Specifically if $\deg \Theta_t = n$ and $I$ denotes the $n\times n$ identity matrix, we can write $M_{\Theta_t} = t I + (1-t^2) A_t$, where 
\begin{equation} \label{eqn:At} M_{\Theta_t} =\begin{bmatrix} t &(1-t^2) & -t(1-t^2) & \dots & (-t)^{n-2}(1-t^2) \\
  & t & (1-t^2) & \ddots & \vdots \\
 & &   & \ddots& -t(1-t^2) \\
 &  & &   & (1-t^2)\\ 
0 & &  &   & t
\end{bmatrix} \ \text{ and }  \ A_t = \begin{bmatrix} 0 &1 & -t & \dots & (-t)^{n-2} \\
  &  0 & 1 & \ddots & \vdots \\
 & &   & \ddots& -t \\
 &  & &   & 1\\ 
0 & &  &   & 0
\end{bmatrix}. \end{equation}
Here, $A_t$ is a nilpotent matrix of order $n$; these $A_t$ matrices are also called  KMS matrices and their numerical ranges have been studied by Gau and Wu in \cite{GauWu13, GauWu14}. Then any continuous function $\vec{x}(s)$ of form
\[ \vec{x}(s) =\begin{bmatrix}x_1(s) \\ \vdots \\ x_n(s)
\end{bmatrix}:[a,b] \subseteq \mathbb{R} \rightarrow \mathbb{S}^{n},\]
where $\mathbb{S}^n$ is the unit sphere in $\mathbb{C}^n$, give a resulting curve of points 
\begin{equation}\label{eqn:curvef} \left \langle A_t \vec{x}(s), \vec{x}(s) \right \rangle, \ \  s \in[ a,b]\end{equation}
in the numerical range $W(A_t)$. Using a carefully chosen $\vec{x}$, \cite{BG21} 
obtained a good approximating curve for the boundary of $W(A_t)$ and used it to both prove Theorem \ref{thm:pshds} and show that for $n=3,4,5$ there are matrices $X_t$ such that for all polynomials $p$,
\begin{equation} \label{eqn:cc2} \| p(A_t) \|  \le \| X_t \|  \|X_t^{-1}\| \max_{z\in W(A_t)} |p(z)|.\end{equation}
Moreover, using Mathematica estimates, the authors concluded that
 \begin{itemize}
 \item If $n=4$ and $t \in (0, 0.42)$, $\| X_t \|  \|X_t^{-1}\| \le 2,$
 \item  If $n=5$ and  $t \in (0.0001, 0.5)$,  $\| X_t \|  \|X_t^{-1}\| \le 2$,\end{itemize}
 which implies that for those $t$ ranges, Crouzeix's conjecture should hold for the associated $A_t$ matrices and hence for the $M_{\Theta_t}$ matrices. This current paper generalizes these results from \cite{BG21} in a number of ways.

\subsection{Main Results \& Paper Overview}

In this paper, we extend and generalize the results from Section 5 in \cite{BG21} discussed above by using carefully-chosen curves to approximate key numerical range boundaries.


First, in Section \ref{sec:uc}, we generalize Theorem \ref{thm:pshds} in two ways. In particular, we both study $\Theta$ of higher degree and look at $\Theta$ that no longer have a single repeated zero. Two of our main results are encoded in the following theorem, which appears later as Theorem \ref{thm:55} and Theorem \ref{thm:44}:
\begin{theorem} \label{thm:main} Let $t\in [0,1)$. Then $\Theta$ satisfies the pseudohyperbolic disk criterion if 
\[ 
\Theta(z) = \left( \frac{z-t}{1-tz}\right)^5 \ \ \text{ or } \ \ \Theta(z) = \left( \frac{z-t}{1-tz}\right)^3 \left( \frac{z-\sqrt{t}}{1-\sqrt{t}z}\right).\]
\end{theorem}

This theorem immediately implies that the associated $M_\Theta$ matrices satisfy the level set Crouzeix conjecture. It is also somewhat surprising. Indeed,  work in \cite{BG21} suggested that the pseudohyperbolic disk criterion would be challenging to prove in the $n=5$ setting. Similarly, Example $6.1$ in \cite{BG21} showed that even simple degree-$2$ $\Theta$ can fail the pseudohyperbolic disk criterion. In contrast, Theorem \ref{thm:main} shows that fairly large classes of $\Theta$ (including $\Theta$ that do not just have a repeated zero) can still satisfy it.

Section \ref{sec:uc} includes additional information and results. Subsection \ref{sec:phd} gives an overview of the techniques (including curve construction) used throughout the remainder of the section. Subsection \ref{subsec:uni} includes both the proof of the first half of Theorem \ref{thm:main} and  studies $\Theta_t$ of form \eqref{eqn:theta} with $n=6,7,8,11.$ Here, strong numerical evidence suggests that for $t\in[0,1),$ all such $\Theta_t$ should satisfy the pseudohyperbolic disk criterion and hence, the level set Crouzeix conjecture. Meanwhile, Subsection \ref{subsec:3x3}  studies $\Theta$ with more than one (repeated) zero and specifically, considers the pseudohyperbolic disk criterion for degree-$3$ Blaschke products $\Phi_t$ with zeros at $t,t, t^{1/m}.$  Theorem \ref{thm: general} shows that for each $m \in \mathbb{N}$, there is a $t_m$ so that if $t\in (t_m,1)$, then $\Phi_{t}$ satisfies the pseudohyperbolic disk criterion and Theorem \ref{thm:n27} obtains explicit values for $t_m$ for the cases $m=2, \dots, 7$. 
These $t_m$ values are not optimal and heavily depend on the vectors used to generate the approximating curve; Remark \ref{rem:curves} includes a discussion of this phenomena and a selection of different vectors/curves adapted to different values of $m$. Subsection \ref{subsec:4x4} contains the proof of the second half of Theorem \ref{thm:main}.

In Section \ref{sec:At}, we investigate Crouzeix's conjecture directly for related nilpotent  matrices. Subsection \ref{sec:nil0} gives an overview of the key tools and techniques we need. In Subsection \ref{subsec:nil1}, we study the $A_t$ matrices in \eqref{eqn:At} and prove the following result.
\begin{theorem} For $n=4$ and $t\in [0,0.363]$, the matrix $A_t$ in \eqref{eqn:At} satisfies Crouzeix's conjecture. \end{theorem} 
This appears later as Corollary \ref{cor:44} and follows from various results we prove about the $X_t$ matrices appearing in \eqref{eqn:cc2}. These results turn the computational work from \cite{BG21} discussed after \eqref{eqn:cc2} into an analytic result. In Remark \ref{rem:68}, we do further computational work in the $n=6,7,8$ cases, which yields $t$-intervals where the associated $A_t$ matrices should satisfy Crouzeix's conjecture. Finally, we note that these arguments apply to related nilpotent matrices as well. Specifically, in Subsection \ref{subsec:nil2}, we apply them to $4\times 4$ nilpotent matrices of the form 
$$A_{t,m} = \left(
\begin{array}{cccc}
 0 & 1 & t & t^m \\
 0 & 0 & 1 & t \\
 0 & 0 & 0 & 1 \\
 0 & 0 & 0 & 0 \\
\end{array} \right)$$
for $m=2,3,4$ and obtain a number of results. For example, we prove that if $m=4$, $A_{t,m}$ satisfies Crouzeix's conjecture for $ t \in [0, 0.367]$. We also look at higher dimensional generalizations of these $A_{t,m}$ matrices. We expect that similar arguments could apply to larger classes of matrices and urge the interested reader to look for additional applications.

\section*{Acknowledgments}
Special thanks to Pamela Gorkin for useful insights and discussions during the writing of this paper. 
Additionally, the authors Kelly Bickel, Georgia Corbett, Annie Glenning, and Changkun Guan were partially supported by the National Science Foundation DMS grant \#2000088 during these research activities. Additional support was provided to Corbett, Glenning, and Guan by Bucknell University.


\section{Level Set Crouzeix Conjecture for Classes of $M_\Theta$} \label{sec:uc}
In this section, we use curves approximating the boundary of the numerical range, denoted $\partial W(M_\Theta),$ to show that several classes of $\Theta$ (each class parameterized by $t\in[0,1)$) satisfy the pseudohyperbolic disk criterion. Then $M_\Theta$ immediately satisfies the level set Crouzeix conjecture.

\subsection{Overview of Approach}  \label{sec:phd} As discussed near \eqref{eqn:curvef}, we 
can construct curves in a numerical range $W(A)$ by choosing a function  $\vec{x}(s): [a,b] \rightarrow \mathbb{S}^n$ and considering the  points $\left \langle A \vec{x}(s), \vec{x}(s) \right \rangle.$ 
For the $n\times n$ matrix $A_t$ as in \eqref{eqn:At}, we follow the work in \cite{BG21, HaagerupdelaHarpe} and often use  $\vec{x}(s)$ on $[0, 2\pi]$  defined component-wise by
\begin{equation} \label{eqn:vec1} x_\ell(s) = \sqrt{\tfrac{2}{n+1}} \sin\left( \tfrac{ \ell \pi}{n+1} \right) e^{i(\ell-1)s}, \quad \text{ for } 1 \le \ell \le n.\end{equation}
In \cite[Proposition 5.1]{BG21}, the authors studied the curve $C_t$ given by  $\left \langle A_t \vec{x}(s), \vec{x}(s) \right \rangle$ from \eqref{eqn:At} and $\vec{x}$ from \eqref{eqn:vec1} and showed that $C_t$ is parameterized by the function
\begin{equation} \label{eqn:Ct1} f_n(s) :=  \left \langle A_t \vec{x}(s), \vec{x}(s) \right \rangle= \sum_{k=1}^{n-1} a_{n,k} (-t)^{k-1}  e^{isk}, \qquad  s \in [0, 2\pi),\end{equation}
where each 
\[ a_{n,k} = \frac{1}{(n+1)\sin\left(\frac{\pi}{n+1}\right)}\left( (n-k) \cos\left( \tfrac{k \pi}{n+1}\right)\sin\left(\tfrac{\pi}{n+1}\right) +  \sin\left( \tfrac{\pi(n-k)}{n+1}\right)\right).\]
Then one can immediately conclude that $W(M_{\Theta_t})$ contains the set of points $t + (1-t^2) C_t$. This was the curve of points used in \cite{BG21} and we often use it in this paper as well. In the cases where $\Theta$ does not have a single repeated zero, $M_\Theta$ does not decompose into the sum of a multiple of $I$ and a nilpotent matrix $A_t$ as in \eqref{eqn:At}. In those settings, we often construct curves that approximate $ \partial W(M_\Theta)$ by using \eqref{eqn:curvef} directly with $M_\Theta$. Additionally, while we often use $\vec{x}(s)$ from \eqref{eqn:vec1},  other times different $\vec{x}$ functions give better approximations of $ \partial W(M_\Theta)$. 

Our proofs that certain $\Theta$ satisfy the pseudohyperbolic disk criterion will requite us to identify a large disk inside of each $W(M_\Theta)$. Our main results in this direction are Theorems \ref{thm:55},  \ref{thm:n27}, \ref{thm: general}, and \ref{thm:44}. Their proofs follow the same four step structure:\\

\begin{itemize}
    \item[1.] Use the construction detailed above to obtain a useful curve $C_t$ in $W(M_{\Theta})$ or $W(A_t)$ parameterized by a polynomial function $f(s)= \sum_{k=0}^n a_k(t) e^{ i\pi k s}$, for real coefficients $a_k(t)$. \\
    
    \item[2.] Find a disk inside of the convex hull of $C_t$ by first identifying a center $c(t)$ for the disk using the formula  $c(t) = \tfrac{ f(0) + f(\pi)}{2}$
     and then finding a radius $R(t)$ such that for all $s\in[0, 2 \pi],$
    \begin{equation} \label{eqn:fcR} |f(s) - c(t) |^2 \ge R(t)^2.\end{equation}
    Since $C_t$ starts at $f(0)$, goes through $f(\pi)$, ends at $f(2\pi)= f(0)$, and is symmetric across the $x$-axis, \eqref{eqn:fcR} implies the Euclidean disk with center $c(t)$ and radius $R(t)$ is in the convex hull of $C_t$. \\
    
    \item[3.] Use Step $2$ to identify a large Euclidean disk in $W(M_\Theta)$.  If $C_t \subseteq W(M_\Theta)$, we use $D(c(t), R(t))$ and if $C_t \subseteq W(A_t)$, we use $D(t +(1-t^2)c(t), (1-t^2) R(t)),$ where we are using the notation for a Euclidean disk from \eqref{eqn:diskform}. Because that Euclidean disk in $\mathbb{D}$ is also a pseudohyperbolic disk, we can find its pseudohyperbolic radius $r(t).$ We then show $r(t)$ is large enough for $\Theta$ to satisfy the pseudohyperbolic disk criterion on a given interval of $t$-values. \\
    \end{itemize}

As an aside, converting a Euclidean disk $D(c,R)$ to a pseudohyperbolic disk $D_\rho(z_0,R)$ requires some work. If $c = 0$, then $D(c,R) = D_\rho(0,R)$. If $c \ne 0$, then $D(c,R)$ is equal to $D_\rho(z_0, r)$, where $r$ is the unique solution in $[0,1)$ of 
\begin{equation} \label{eqn:r}  r + \tfrac{1}{r} = \frac{R^2 -|c|^2 +1}{R}\end{equation}
and $z_0$ satisfies $\arg z_ 0 = \arg c$ and $|z_0|$ is the unique solution in $[0,1)$ of 
\[ |z_0| + \tfrac{1}{|z_0|} = \frac{ |c|^2 -R^2 +1}{|c|}.\]
These formulas appear, for example, in \cite{MortiniRupp2021}.

\subsection{Blaschke products with a repeated zero} \label{subsec:uni}
In this subsection, we consider Blaschke products with a single repeated zero at $t\in [0,1)$.  We first extend Theorem \ref{thm:pshds} to the $n= 5$ case.

\begin{theorem} \label{thm:55}
Let $\Theta_t$ be a  Blaschke product of the form \eqref{eqn:unicritical} with $\deg \Theta_t = 5$. Then $W(M_{\Theta_t})$ always contains a pseudohyperbolic disk of radius $\frac{\sqrt{3}}{2}$.
\end{theorem}

\begin{proof} We follow the structure from Section \ref{sec:phd}. First, using \eqref{eqn:Ct1}, we obtain a curve $C_t$ in $W(A_t)$ where the points on $C_t$ are given by
$f(\text{s})=\frac{\sqrt{3} }{2} e^{i s}-\frac{7 t}{12}e^{i 2 s}+\frac{t^2}{2\sqrt{3}}e^{i 3 s}-\frac{t^3}{12} e^{i 4 s}$ for $s\in [0,2\pi]$. Then we identify the center of a disk in $W(A_t)$ by
\[ c(t)  =\tfrac{ f(0) + f(\pi)}{2}= -\frac{t}{12}\left(t^2+7\right).\]
To find the radius of the disk, we consider
\begin{align} \nonumber
 \left \vert f(s)+\frac{t}{12} \left(t^2+7\right) \right \vert^2 &= \left \vert -\frac{1}{12} t^3 \cos (4 s)+\frac{t^2 \cos (3 s)}{2 \sqrt{3}}-\frac{7}{12} t \cos (2 s)+\frac{1}{2} \sqrt{3} \cos (s)+\frac{1}{12} t \left(t^2+7\right) \right \vert^2 \\
 &+ \left \vert -\frac{1}{12} t^3 \sin (4 s)+\frac{t^2 \sin (3 s)}{2 \sqrt{3}}-\frac{7}{12} t \sin (2 s)+\frac{1}{2} \sqrt{3} \sin (s) \right \vert^2. \label{eqn: function 1} 
\end{align}
Setting $x=\cos(s)$, using the identities
\begin{align*}
    \cos(2s) &= 2\cos ^2(s)-1\\
    \cos(3s) &= 4\cos^3(s)-3 \cos (s)\\
    \cos(4s)&= 8 \cos ^4(s)-8 \cos ^2(s)+1,
\end{align*}
and simplifying, we can conclude that the right side of  \eqref{eqn: function 1} equals
\begin{align*}
    & \left(\frac{t^4}{12}+\frac{t^2}{2}+\frac{3}{4}\right)+\frac{1}{36} t^2 \left(1-x^2\right) \left(4 t^4 x^2+28 t^2 x^2-4 \sqrt{3} t^3 x-16 \sqrt{3} t x+13\right) 
    = \left(\frac{t^4}{12}+\frac{t^2}{2}+\frac{3}{4}\right)+ g(t,x).
\end{align*}
Using standard calculus computations, one can show that $g(t,x)\geq 0$ on $[0,1]\times [-1,1]$. 
and so, we can set $R(t)^2= \frac{t^4}{12}+\frac{t^2}{2}+\frac{3}{4}.$ Then
we have  $$D\left(t-(1-t^2)\frac{t}{12}\left(t^2+7\right),\left(1-t^2\right) \sqrt{\frac{t^4}{12}+\frac{t^2}{2}+\frac{3}{4}}\right )\subseteq W(M_{\Theta_t}).$$
 This Euclidean disk is also a pseudohyperbolic disk $D_\rho(z_0(t), r(t))$. By solving \eqref{eqn:r} for $r(t)$,
we obtain 
\begin{align}
    r(t)= \frac{g_1(t)-\sqrt{g_2(t)}}{g_3(t)},
\end{align}
where 
\begin{align*}
    g_1(t)&= \sqrt{3} t^8+\sqrt{3} t^6-\sqrt{3} t^4+83 \sqrt{3} t^2+252 \sqrt{3} \\
    g_2(t)&=3 t^{16}+6 t^{14}-3 t^{12}+492 t^{10}+2013 t^8+1014 t^6-1581 t^4+1080 t^2+3888\\
    g_3(t)&=144 \left(t^2+3\right).
\end{align*}
One can check that $r(0) = \frac{\sqrt{3}}{2}$ and we will prove $r(t)> \frac{\sqrt{3}}{2}$  for  $t\in (0,1)$.  Proceeding towards a contradiction, 
suppose there exists a $t_*\in (0,1)$ such that $r(t_*)< \frac{\sqrt{3}}{2}$. Since $r(t)$ is continuous on $(0,1)$ and $r(0.5)\approx 0.873 >  \frac{\sqrt{3}}{2},$ the intermediate value theorem gives a $t' \in (0,1)$ with $r(t') =  \frac{\sqrt{3}}{2}$.
Then $t'$ also satisfies
\[ \left( \frac{\sqrt{3}}{2} g_3(t') -g_1(t') \right)^2 = g_2(t').\]
Moving everything to the left side of that equation and simplifying gives $p(t')=0$, where $p$ is the degree $10$ polynomial $$p(t)=-t^{10}-4 t^8-2 t^6+4 t^4+3 t^2.$$ One can easily check that the $10$ zeros of $p$ are $-i,-i, 0,0, i, i$ and $-1,1, -i\sqrt{3}, i \sqrt{3}$. Since we assumed $t'\in (0,1)$ was another zero, this gives our contradiction. Thus, $r(t)> \frac{\sqrt{3}}{2}$ for all $t\in (0,1)$ and so, $W(M_{\Theta_t})$ contains a pseudohyperbolic disk of at least radius $\frac{\sqrt{3}}{2}$.
\end{proof}

Since $\frac{\sqrt{3}}{2}> \left(\frac{1}{2}\right)^{(1/4)}$, Theorem \ref{thm:55} implies that degree-$5$ $\Theta_t$ of the form \eqref{eqn:unicritical} satisfy the pseudohyperbolic disk criterion. This $n=5$ case is also illustrated by Figure \ref{fig:phd}, which shows a particular $W(M_{\Theta_t})$, the shifted curve $t +(1-t)^2 C_t$, and a pseudohyperbolic disk with radius $\left(\frac{1}{2}\right)^{(1/4)}$ inside of that curve.

\begin{figure}[h!] 
    \subfigure[$n=5$]
      {\includegraphics[width=0.42 \textwidth]{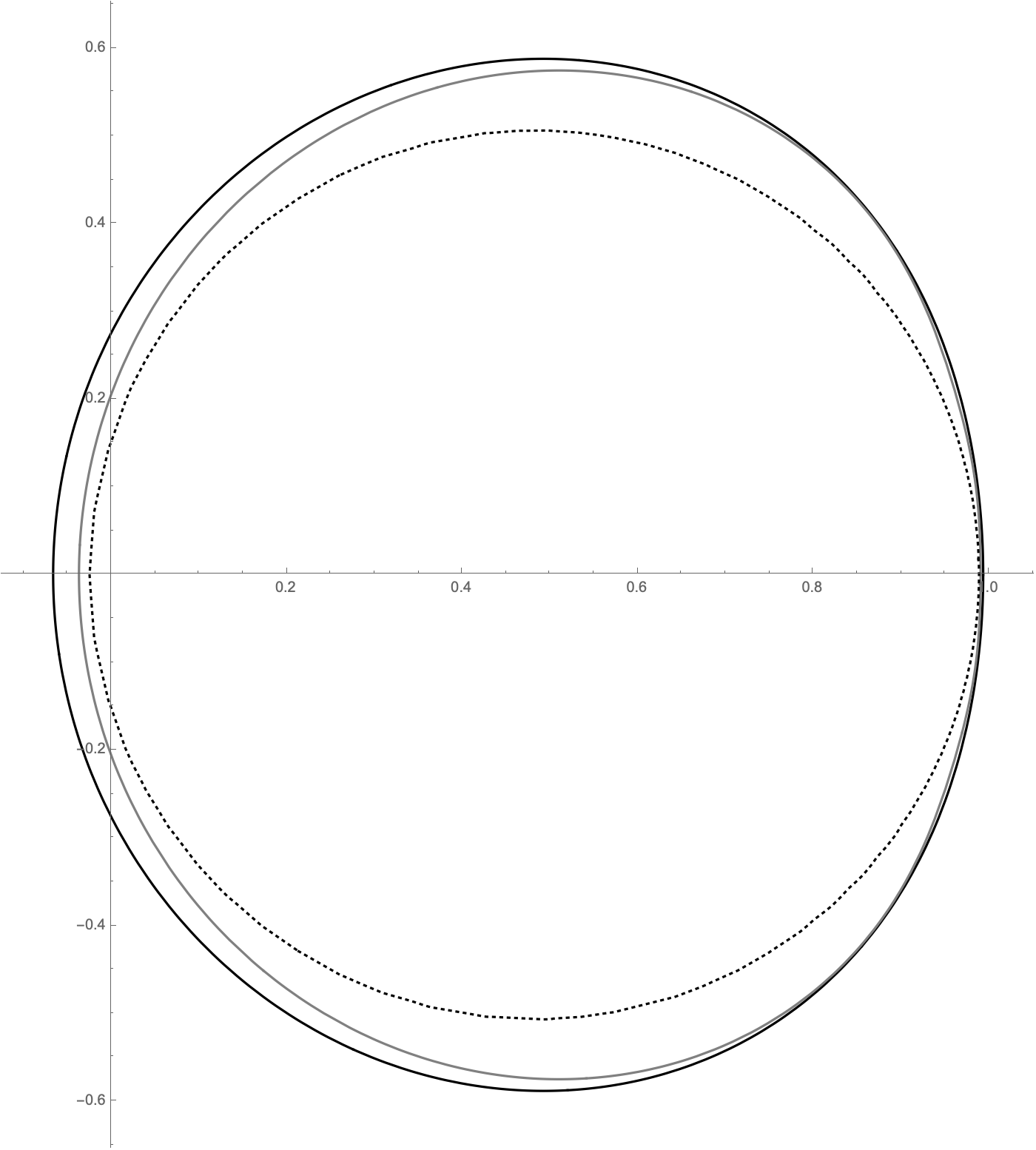}}
    \quad 
    \subfigure[$n=6$]
      {\includegraphics[width=0.42 \textwidth]{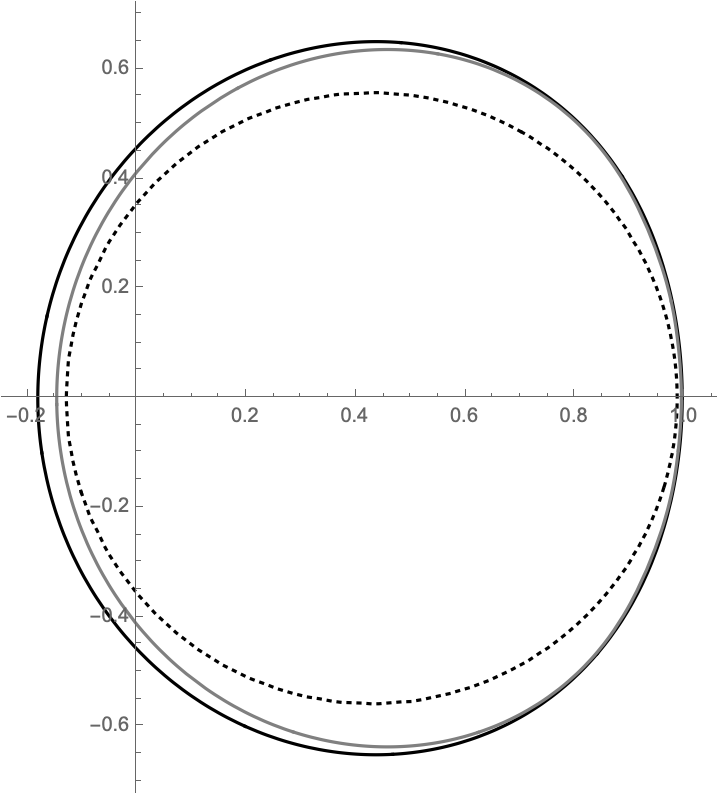}}
  \caption{\textsl{For $t=0.7$ and two values of $n$, $W(M_{\Theta_t})$ (black), the curve $t +(1-t^2)C_t$ (gray), and a pseudohyperbolic disk with radius $(\frac{1}{2})^{1/(n-1)}$ (dashed).}} \label{fig:phd}
\end{figure}

It is worth noting that Theorem \ref{thm:phdcP} gives an immediate corollary of Theorem \ref{thm:55}.

\begin{corollary} Let $\Theta_t$ be a  Blaschke product of the form \eqref{eqn:unicritical} with $\deg \Theta_t = 5$. Then $M_{\Theta_t}$ satisfies the level set Crouzeix conjecture. \end{corollary}

 Figure \ref{fig:phd} also illustrates the $n=6$ case. Based on the figure, it looks like we can a find pseudohyperbolic disk with radius $\left(\frac{1}{2}\right)^{(1/5)}$  inside of $t +(1-t^2)C_t \subseteq W(M_{\Theta_t})$. We explore higher dimensional cases like that in the following remark. We actually look for pseudohyperbolic radii $r= \cos (\frac{\pi}{n+1})$, since that formula aligns with the value from Theorem \ref{thm:55} and is at least $\left(\frac{1}{2}\right)^{1/(n-1)}$ in our setting.

\begin{remark} \label{rmk:11}
Let $\Theta_t$ be a Blaschke product of the form \eqref{eqn:unicritical} with $\deg \Theta_t = n >5$.
In this case, we can still approximate the boundary of $W(A_t)$ with the curve $C_t$ parameterized by $f(s)=f_n(s)$ from \eqref{eqn:Ct1}. In general, \eqref{eqn:fcR} is not easily simplified so we are no longer able to identify a nice radius function $R(t)$ and proceed as in proof for Theorem \ref{thm:55}.
Instead, we define $c(t)$ as before but rephrase the question to ask if a disk with pseudohyperbolic radius $r= \cos (\frac{\pi}{n+1})$ and Euclidean center $\hat{c_t}=t+(1-t^2)c(t)$ is within the convex hull of $t+(1-t^2)C_t$. It turns out that this question can be explored with Mathematica.
 
 We investigated the cases where $n=6,7,8,11$ and in each case, Mathematica computations suggest that $\Theta_t$ should satisfy the pseudohyperbolic disk criterion. Because the formulas involve $\cos \left(\frac{\pi}{n+1}\right)$, the $n=11$ case is actually the simplest and we include those details below.  The other cases are similar.
 
 First, using the curve $C_t$ parameterized by $f(s)=f_{11}(s)$ found in \eqref{eqn:Ct1}, one can deduce that
$$ c(t)= \frac{1}{24} t \left(\left(\sqrt{3}-2\right) t^8+\left(1-2 \sqrt{3}\right) t^6-2 \left(\sqrt{3}+2\right) t^4-\left(2 \sqrt{3}+11\right) t^2-11 \sqrt{3}-2\right).$$
Then if we consider a disk with Euclidean center $\hat{c_t}=t+(1-t^2)c(t)$ and pseudohyperbolic radius $r(t)= \cos (\frac{\pi}{12})$, one can use  \eqref{eqn:r} to solve for the Euclidean radius of that disk. That gives 
 \begin{align*}
     R_{11}(t)=& \frac{36 \sqrt{2}+6 \sqrt{6}-\sqrt{h(t)}}{24 \left(\sqrt{3}+1\right)}
 \end{align*}
where 
\begin{align*}
    h(t)=& \left(4-2 \sqrt{3}\right) t^{22}+\left(12 \sqrt{3}-12\right) t^{20}+56 t^{18}+\left(60 \sqrt{3}+88\right) t^{16}+\left(170 \sqrt{3}+148\right) t^{14}\\
    &+\left(96 \sqrt{3}+584\right) t^{12}+\left(410 \sqrt{3}+244\right) t^{10}+\left(252 \sqrt{3}+472\right) t^8+632 t^6\\
    &+\left(396 \sqrt{3}-396\right) t^4+\left(484-242 \sqrt{3}\right) t^2-288 \sqrt{3}+504.
\end{align*}
We wish to show that the  disk with pseudohyperbolic radius $r= \cos (\frac{\pi}{12})$ and Euclidean center \\$\hat{c_t}=t+(1-t^2)c(t)$ is within the convex hull of $t+(1-t^2)C_t$.
That will follow if we show that \eqref{eqn:fcR} holds with $R(t)=\frac{R_{11}(t)}{1-t^2}$. 
Using the Mathematica Minimize command and corresponding plots, one can obtain strong evidence to suggest that this inequality holds for $0\leq t\leq 0.99$. Indeed, this inequality appears to hold for $t\in [0,1)$, but the Mathematica Minimize command becomes less stable as $t$ approaches $1$ because we are dividing by $1-t^2$. So when $n=11$, this evidence indicates that  $W(M_{\Theta_t})$ contains a pseudohyperbolic disk with radius $\cos\left (\frac{\pi}{12}\right )$. We observed analogous behavior for the cases $n=6,7,8.$
\end{remark}

Theorem \ref{thm:55} coupled with Remark \ref{rmk:11} suggests the following conjecture.

\begin{conjecture} If $t \in [0, 1)$ and $\Theta_t$ is of form \eqref{eqn:unicritical} for any value of $n \ge 3$, then $W(\Theta_t)$ contains a pseudohyperbolic disk of radius $\cos\left( \frac{\pi}{n+1}\right)$.
\end{conjecture}

Since $\left(\frac{1}{2}\right)^{1/(n-1)} < \cos\left( \frac{\pi}{n+1}\right)$ for each $n \ge 3$, this conjecture would imply that such $\Theta_t$ satisfy the pseudohyperbolic disk criterion and hence, the level set Crouzeix conjecture.


\subsection{Other Blaschke products: the $3 \times 3$ case} \label{subsec:3x3}
In the next two subsections, we show that several classes of Blaschke products beyond those studied in \cite{BG21} also satisfy the pseudohyperbolic disk criterion and hence, the level set Crouzeix conjecture.

In this subsection, we consider functions of the following form, with a repeated zero at $t \in[0,1)$ and an additional zero at $t^{1/m}$, for some positive integer $m$:
\begin{equation} \label{eqn:thetat2}  \Phi_t(z):= \left(\frac{z-t}{1-t z}\right)^2\left(\frac{z-t^{1/m}}{1-t^{1/m}z}\right).\end{equation}

For such functions, we have the following result. 

\begin{theorem} \label{thm:n27}
Let $\Phi_t$ be a degree-three Blaschke product of the form \eqref{eqn:thetat2}. Then for each integer $m$ with $2 \le m \le 7$, there is a specific $t_m \in [0,1)$ such that for each $t\in (t_m,1)$, $W(M_{\Phi_t})$ contains a pseudohyperbolic disk with radius $\frac{1}{\sqrt{2}}$. These values of $t_m$ are given in the following table:
\[
\renewcommand*{\arraystretch}{1.2}
\begin{tabular}{||c | l||} 
 \hline
 \hspace{0.1cm} $m$  \hspace{0.1cm} & \hspace{0.25cm} $t_m$ \hspace{0.25cm} \\ [0.5ex]
 \hline
  2 &  \hspace{0.25cm} 0.11  \hspace{0.25cm}\\ 
 \hline
 3 &\hspace{0.25cm} 0.27\hspace{0.25cm} \\ 
 \hline
 4 & \hspace{0.25cm} 0.41 \hspace{0.25cm}\\ 
 \hline
 5 &\hspace{0.25cm} 0.50\hspace{0.25cm} \\ 
 \hline
 6 & \hspace{0.25cm} 0.57 \hspace{0.25cm}\\ 
 \hline
 7 & \hspace{0.25cm} 0.62 \hspace{0.25cm}\\ 
 \hline
\end{tabular}
\]
\end{theorem}

\begin{proof}
We follow the general argument from Subsection \ref{sec:phd}. Let $f(s) = \langle M_{\Phi_t} \vec{x}(s), \vec{x}(s)\rangle$, where $\vec{x}$ is the vector-valued function from \eqref{eqn:vec1} when $n=3$. Then $f$ is given by
\[f(s) = \frac{1}{4} \left(-t e^{2 i s} \sqrt{1-t^2}  \sqrt{1-t^{2/m}}+\sqrt{2} e^{i s} \left(\sqrt{1-t^2} \sqrt{1-t^{2/m}}-t^2+1\right)+t^{1/m}+3 t\right)\]
 and parameterizes a curve $C_t$ that approximates the boundary of $W(M_{\Phi_t})$. Then we can identify the center of a disk in $W(M_{\Phi_t})$ by
\[ c_m(t)  =\tfrac{ f(0) + f(\pi)}{2} =\frac{1}{4} \left(t^{1/m}+t \left(3-\sqrt{1-t^2} \sqrt{1-t^{2/m}}\right)\right).\]
 A simple computation gives 
  \[
  \begin{aligned}
      |f(s)-c_m(t)|^2  &= \frac{1}{8} \left(1-t^2\right) \left(2 \sqrt{1-t^2} \sqrt{1-t^{2/m}}+2-t^2 \left(1-t^{2/m}\right) \cos (2 s)-\left(t^2+1\right) t^{2/m}\right)\\
       &\geq \frac{1}{8} \left(1-t^2\right)  \left(2 \sqrt{1-t^2} \sqrt{1-t^{2/m}}+2-t^2\left(1-t^{2/m}\right)-\left(t^2+1\right) t^{2/m}\right) \\
       &=\frac{1}{8} \left(1-t^2\right)  \left(2 \sqrt{1-t^2} \sqrt{1-t^{2/m}}+2 -t^2-t^{2/m} \right)
  \end{aligned}
  \]

 By the arguments in Subsection \ref{sec:phd}, the Euclidean disk $D(c_m(t), R_m(t))$ with center $c_m(t)$ and radius $R_m(t)$ given by
 $$R_m(t)=\sqrt{\frac{1}{8} \left(1-t^2\right)  \left(2 \sqrt{1-t^2} \sqrt{1-t^{2/m}}+2 -t^2-t^{2/m} \right)}$$
is in $W(M_{\Phi_t}).$  This Euclidean disk is also a pseudohyperbolic disk whose pseudohyperbolic radius $r_m(t)$ satisfies \eqref{eqn:r} with $r_m \in[0,1)$. As it is rather complicated, we do not include the formula for $r_m(t)$ here, but it is easy to check that since $R_m$ and $c_m$ are continuous on $[0,1)$, $r_m(t)$ is as well.
 
We wish to show that for each $m$, $r_m(t) \ge \frac{1}{\sqrt{2}}$ for $t\in (t_m,1)$, where $t_m$ is given in the table above. As the arguments are very similar for each $m$ value, we only provide the details for $m=2$.
By contradiction, suppose there exists a $t_*\in (0.11,1)$ such that $r_2(t_*)<\frac{1}{\sqrt{2}}$. As 
 $r_2(t)$ is continuous on $(0,1)$ and $r(\tfrac{1}{2})> \tfrac{1}{\sqrt{2}},$ the intermediate value theorem gives a $t' \in (0.11,1)$ with $r_2(t') = \frac{1}{\sqrt{2}}$.
We now find a polynomial $p$ such that this implies $p(\sqrt{t'})=0$; by locating the zeros of $p$ and seeing that they are all outside of $(\sqrt{0.11},1)$, we will obtain our contradiction. 

To that end, using \eqref{eqn:r}, one can conclude that $\sqrt{t'}$ is a solution to $$\left(\sqrt{2}+\frac{1}{\sqrt{2}}\right)R_2(t^2)=R_2(t^2)^2-c_2(t^2)^2+1.$$
Squaring both sides of the equation, we have that $\sqrt{t'}$ is also a solution to $$\left(\left(\sqrt{2}+\frac{1}{\sqrt{2}}\right)R_2(t^2)\right)^2=\left(R_2(t^2)^2-c_2(t^2)^2+1\right)^2.$$ 
The only remaining square root terms are those with a factor of $\sqrt{1+t^2}$. Isolating those terms on the right, squaring both sides again, and moving everything to the right gives a degree 40 polynomial $p(t)$ with $p(\sqrt{t'})=0$. After simplifying $p$, we have 
\[ \begin{aligned}
p(t)&= -16(t-1)^4(1+t^2)(-32 - 64t - 70t^2 - 68t^3 + 57t^5 + 51t^6 + 26t^7 - 6t^8\\
 - &13t^9 - 9 t^{10} - 5 t^{11} + t^{12} + 3 t^{13} + t^{14})^2+(128 + 8 t^2 - 224 t^3 - 199 t^4
 +20 t^5 + 112 t^6 \\
 + & 184 t^7 + 
    198 t^8 - 28 t^9 - 146 t^{10} - 44 t^{11}- 81 t^{12} + 20 t^{13}
    +54 t^{14}+ 3 t^{16} - 6 t^{18} + t^{20})^2.
 \end{aligned} \]
Mathematica shows that, of the 40 zeros of $p$, the only zero in the range $(0,1)$ is approximately $0.319.$ However, our assumption that $t'\in (0.11,1)$ implies the zero $\sqrt{t'}\in (0.33,1)$, which gives the contradiction. Therefore, for all $t\in (0.11,1)$, we have that $r(t)>\frac{1}{\sqrt{2}}$ as required. 
\end{proof}

\begin{figure}[h!] 
    \subfigure[$m=2$]
      {\includegraphics[width=0.42 \textwidth]{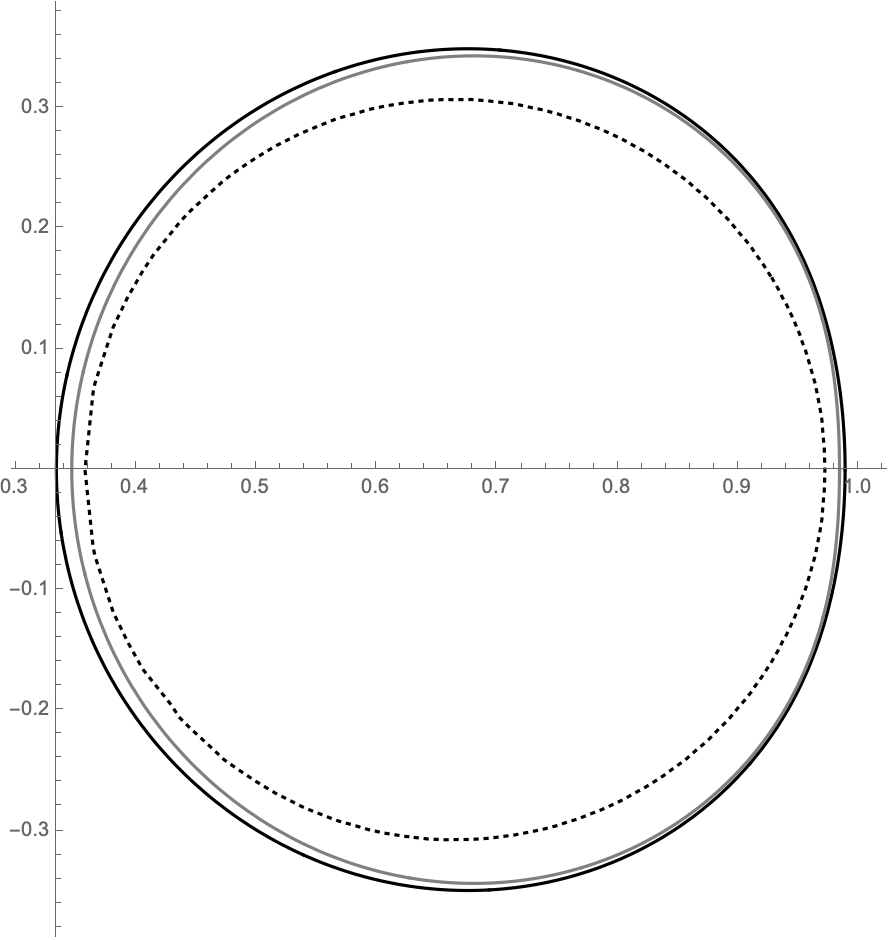}}
    \quad 
    \subfigure[$m=7$]
      {\includegraphics[width=0.45 \textwidth]{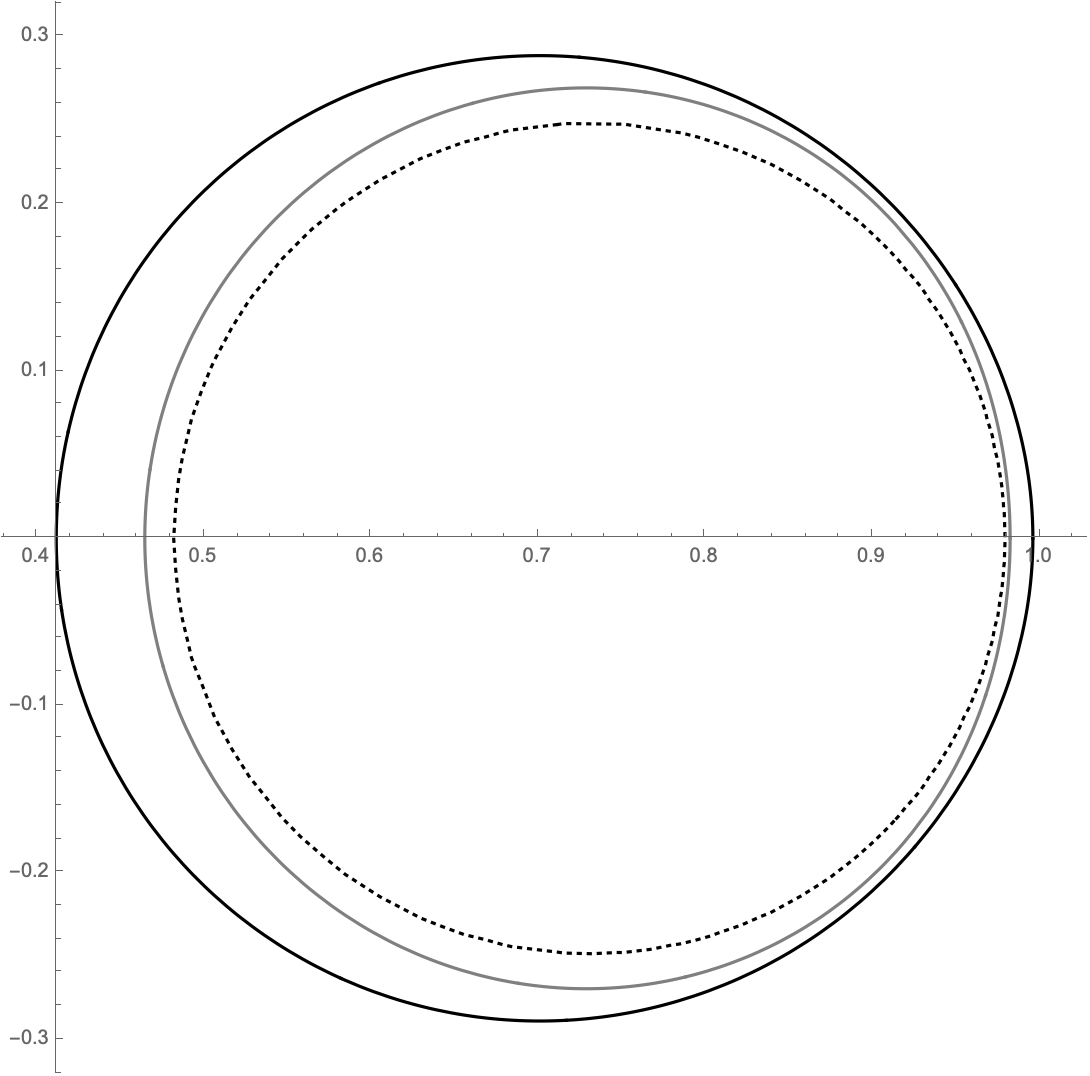}}
  \caption{\textsl{For $t=0.7$, $\partial W(M_{\Phi_t})$ (black), the curve $C_t$ (gray), a pseudohyperbolic disk with radius $\frac{1}{\sqrt{2}}$ in $C_t$ (dashed).}}\label{fig:n27}
\end{figure}

Theorem \ref{thm:n27} is illustrated in Figure \ref{fig:n27}. Here, one can see that when $m=2$, the curve $C_t$ does a much better job approximating $\partial W(M_{\Phi_t})$ than when $m=7$. That partially explains why the interval in Theorem \ref{thm:n27} is much larger for $m=2$ than for $m=7.$

In Theorem \ref{thm:n27}, we used the standard formula from \cite{BG21} to generate the $C_t$ curves. However, for each $m$, we can actually find curves that works better for this particular problem, i.e. curves that include larger pseudohyperbolic disks in their convex hulls.  The following remark gives the details.

\begin{remark} \label{rem:curves}  To obtain the $t_m$ values in Theorem \ref{thm:n27}, we made curves $C_t \subseteq W(M_{\Phi_t})$ using \eqref{eqn:vec1} for $n=3$ and showed that for $t\in (t_m,1)$, the curve $C_t$ contains a pseudohyperbolic disk of radius $\frac{1}{\sqrt{2}}$ in its convex hull. Those specific values of $t_m$ depended on the chosen curve.

However, as discussed in Subsection \ref{sec:phd} and the introduction, we can generate a curve inside of the associated numerical range using any continuous vector-valued function $\vec{x}: [a,b] \rightarrow \mathbb{S}^n$. For example, consider functions of the form
$$\vec{x}(s)=\begin{pmatrix}
a \\ be^{is} \\ ce^{2is}
\end{pmatrix},$$
where $a,b,c$ are real constants satisfying $a^2+b^2+c^2=1$. We obtain \eqref{eqn:vec1} when $n=3$ by setting $a=1/2$, $b=1/\sqrt{2}$, and $c=1/2$. 

More generally, each choice of $a,b,c$ gives a different curve of points $C_t$ in $W(M_{\Phi_t})$ via the formula $\langle M_{\Phi_t} \vec{x}(s), \vec{x}(s)\rangle$. By changing these constants, we can investigate a bunch of different curves for each $\Phi_t$ and try to find optimal ones. Using this method paired with analyses analogous to those in the proof of Theorem \ref{thm:n27}, for each $m\in \{2, \dots, 7\}$, we found particularly good values of $a,b,c$ so that the associated curve $C_t$ contained a pseudohyperbolic disk with radius $\frac{1}{\sqrt{2}}$ for all $t\in (t^*_m, 1)$ where $t_m^*$ is smaller (and often significantly smaller) than the $t_m$ value in Theorem \ref{thm:n27}. Thus, the $\Phi_t$ in Theorem \ref{thm:n27} actually satisfy the pseudohyperbolic disk criterion for $t$ in larger intervals that those indicated in the theorem. For each value of $m$, Figure \ref{fig:newtm} below contains both the vector generating the new curve $C_t$ and the improved value $t_m^*$. 

\begin{figure}[h!]
    \centering
\begin{centering}
\renewcommand*{\arraystretch}{2}
\begin{tabular}{||c | c | c | c | c||}
 \hline
 \hspace{0.1cm} $m$  \hspace{0.1cm} & \hspace{0.25cm} $a$ 
 &\hspace{0.25cm} $b$ 
 &\hspace{0.25cm} $c$ 
 &\hspace{0.25cm} $t^*_m$
 \hspace{0.25cm} \\  [0.5ex]
 \hline
 \hspace{0.1cm} $2$  \hspace{0.1cm} & \hspace{0.25cm} $\frac{\sqrt{6}}{5}$ 
 &\hspace{0.25cm} $\frac{4\sqrt{19}}{25}$ 
 &\hspace{0.25cm} $\frac{3\sqrt{19}}{25}$ 
 &\hspace{0.25cm} $0.09$
 \hspace{0.25cm} \\
 \hline
 \hspace{0.1cm} $3$  \hspace{0.1cm} & \hspace{0.25cm} $\frac{1}{\sqrt{5}}$ 
 &\hspace{0.25cm} $\frac{2}{3}$ 
 &\hspace{0.25cm} $\frac{4}{3\sqrt{5}}$ 
 &\hspace{0.25cm} $0.18$ 
 \hspace{0.25cm} \\
 \hline
 \hspace{0.1cm} $4$  \hspace{0.1cm} & \hspace{0.25cm} $\frac{\sqrt{2}}{\sqrt{11}}$ 
 &\hspace{0.25cm} $\frac{3\sqrt{6}}{11}$ 
 &\hspace{0.25cm} $\frac{3\sqrt{5}}{11}$ 
 &\hspace{0.25cm} $0.24$ 
 \hspace{0.25cm} \\
 \hline
 \hspace{0.1cm} $5$  \hspace{0.1cm} & \hspace{0.25cm} $\frac{\sqrt{3}}{2\sqrt{5}}$ 
 &\hspace{0.25cm} $\frac{\sqrt{17}}{5\sqrt{2}}$ 
 &\hspace{0.25cm} $\frac{\sqrt{51}}{10}$ 
 &\hspace{0.25cm} $0.28$ 
 \hspace{0.25cm} \\
 \hline
 \hspace{0.1cm} $6$  \hspace{0.1cm} & \hspace{0.25cm} $\frac{1}{2\sqrt{2}}$ 
 &\hspace{0.25cm} $\frac{3\sqrt{7}}{10\sqrt{2}}$ 
 &\hspace{0.25cm} $\frac{\sqrt{14}}{5}$ 
 &\hspace{0.25cm} $0.28$ 
 \hspace{0.25cm} \\
 \hline
 \hspace{0.1cm} $7$  \hspace{0.1cm} & \hspace{0.25cm} $\frac{1}{3}$ 
 &\hspace{0.25cm} $\frac{2\sqrt{2}}{3\sqrt{3}}$ 
 &\hspace{0.25cm} $\frac{4}{3\sqrt{3}}$ 
 &\hspace{0.25cm} $0.29$ 
 \hspace{0.25cm} \\
 \hline
\end{tabular}
\end{centering}
    \caption{Summary of new curves and improved $t^*_m$ values for $\Phi_t$ from \eqref{eqn:thetat2} with $m=2,...,7.$}
    \label{fig:newtm}
\end{figure}

It is worth noting that these values of $a,b,c$ are probably not the most optimal vectors for their respective $m$ values, and further investigation could reveal that the $\Phi_t$ satisfy the pseudohyperbolic disk criterion on even larger $t$-intervals.  Additionally, these new curves do not appear to generally enclose larger areas than the original curves. Rather, these new curves yield improved $t^*_m$ values because they appear to be closer to the edge of the unit disk, which should generally increase the pseudohyperbolic radius of an enclosed disk. Figure \ref{fig:newcurve} gives both the previous and new curves for $\Phi_t$ with $t=0.4$ for both $m=4$ and $m=7$. One can see that, in both cases, the new curves are closer to the edge of the unit disk than the original curves.

\begin{figure}[h!] 
    \subfigure[$m=4$]
      {\includegraphics[width=0.428 \textwidth]{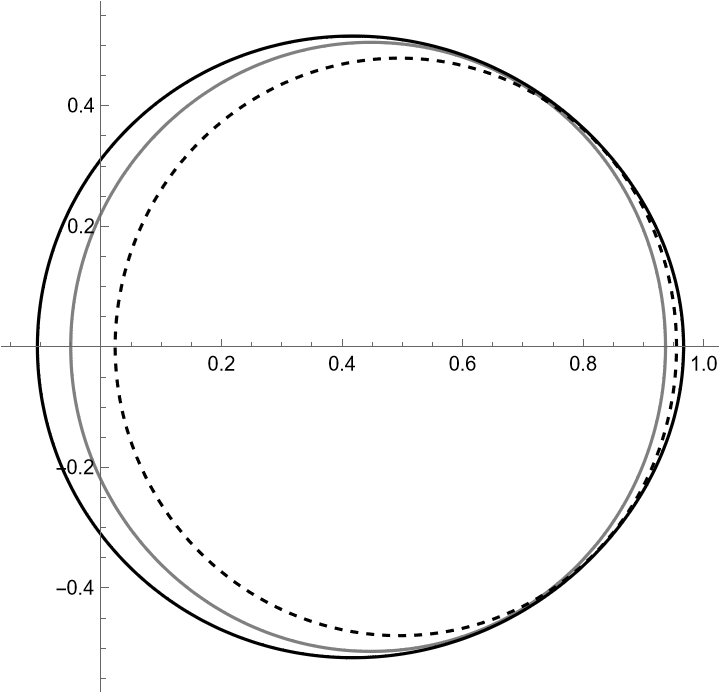}}
    \quad 
    \subfigure[$m=7$]
      {\includegraphics[width=0.45 \textwidth]{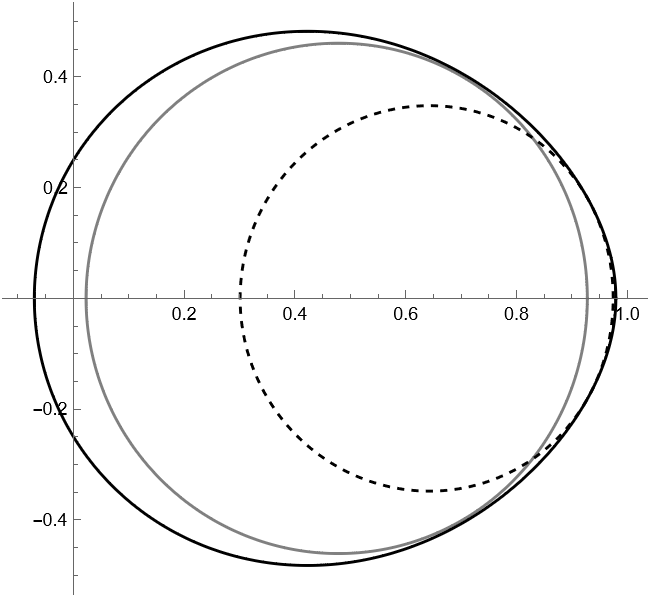}}
  \caption{\textsl{For $t=0.4$, $\partial W(M_{\Phi_t})$ (black), the original curve $C_t$ (gray), and the new curve (dashed)}.}\label{fig:newcurve}
\end{figure}

\end{remark}


In Theorem \ref{thm:n27} and Remark \ref{rem:curves}, we restricted to $m$ values between $2$ and $7$ because each $m$ value required its own sequence of computations to yield the interval endpoints $t_m$ or $t_m^*$. However, the following result shows that, even if a general formula is beyond the scope of this paper, there does exist a cut-off value $t_m$ for each $m$ in the following sense: the pseudohyperbolic disk criterion is satisfied for $\Phi_t$ of form \eqref{eqn:thetat2} with $t \in(t_m,1)$.

\begin{theorem} \label{thm: general}
Let $\Phi_t$ be a degree-three Blaschke product of the form \eqref{eqn:thetat2}.  For each $m \in \mathbb{N}$ with $m >1$, there exists a $t_m \in (0,1)$, such that for $t\in (t_m,1)$, $W(M_{\Phi_t})$ contains a pseudohyperbolic disk with radius $\frac{1}{\sqrt{2}}$. 
\end{theorem}
\begin{proof}
This proof has 3 steps.\\

\noindent \textbf{Step 1: Finding the equation for $r(t)$}. We initially proceed as in Subsection \ref{sec:phd} and define the boundary approximating curve $C_t$ using $\langle M_{\Phi_t} \vec{x}(s), \vec{x}(s)\rangle$, where $\vec{x}$ is given by
\[
\renewcommand*{\arraystretch}{1.2}
  \vec{x}(s)=
\begin{pmatrix}
\frac{\sqrt{11}}{6}\\
\frac{2}{3}e^{is}\\
\frac{1}{2}e^{i2s}\\
\end{pmatrix} \text{  for  } s\in [0,2\pi).
\]
 Then the points on $C_t$ are given by the function
    $$f(s) = \frac{1}{36} \left(9 \left(t^{1/m}+3 t\right) + 4 \left(3 \sqrt{1-t^2} \sqrt{1-t^{2/m}}-\sqrt{11} t^2+\sqrt{11}\right)e^{i s}-\left(3 \sqrt{11} t \sqrt{1-t^2} \sqrt{1-t^{2/m}}\right)e^{2 i s}\right),$$
and the formula for the center of the proposed disk in the convex hull of $C_t$ is
$$c(t)=\frac{f(\pi)+f(0)}{2}=\frac{1}{12}(3t^{1/m}+t\left (9-\sqrt{11}\sqrt{1-t^2}\sqrt{1-t^{2/m}})\right ).$$ A simple computation gives 
  \[
  \begin{aligned}
      |f(s)-c(t)|^2 
      & \geq \frac{1}{324}\left( 1 - t^2\right )\left (80 + 
   24\sqrt{11} \sqrt{1 - t^2} \sqrt{1 - t^{2/m}}- 36 t^{2/m}  - 
   44 t^2\right ).
  \end{aligned}
  \]
This implies that the Euclidean disk $D(c(t), R(t)) \subseteq W(M_{\Phi_t})$, where $$R(t)=\sqrt{\frac{1}{324}\left( 1 - t^2\right )\left (80 + 
   24\sqrt{11} \sqrt{1 - t^2} \sqrt{1 - t^{2/m}}- 36 t^{2/m}  - 
   44 t^2\right )}.$$
Using \eqref{eqn:r}, we can find the pseudohyperbolic radius $r(t)$ of this disk 
\begin{align} \label{eqn:r(t)}
    r(t)= \frac{1-c(t)^2+R(t)^2-\sqrt{-4 R(t)^2+(-1+c(t)^2-R(t)^2)^2}}{2R(t)}.
\end{align}

\noindent
\textbf{Step 2: Finding the limit of $r(t)$}. As part of the proof, we need to compute the $\lim_{t \to 1^-}r(t)$. It is immediate that 
\begin{align} \label{eqn:r(t^2)}
 \lim_{t \to 1^-}r(t) = \lim_{t \to 1^-}r(t^m)= \lim_{t \to 1^-}\frac{1-c(t^m)^2+R(t^m)^2-\sqrt{-4 R(t^m)^2+(-1+c(t^m)^2-R(t^m)^2)^2}}{2R(t^m)},
\end{align}
and so we will compute that final limit instead.
Since $R(1)=0$, we need to do some initial algebraic simplification. By multiplying the numerator and denominator by $\frac{1}{1-t}$, we can find the desired limit by computing these individual limits 
    \begin{align}
        & \lim_{t\rightarrow 1^-} \frac{2R(t^m)}{1-t}  \label{lim:1}\\
        & \lim_{t\rightarrow 1^-} \frac{1-c(t^m)^2+R(t^m)^2}{1-t}  \label{lim:2} \\ 
        & \lim_{t\rightarrow 1^-} \frac{\sqrt{-4 R(t^m)^2+(-1+c(t^m)^2-R(t^m)^2)^2}}{1-t}  \label{lim:3}
    \end{align}
and inserting them back into \eqref{eqn:r(t^2)}.

We will first compute $\eqref{lim:1}$.   First notice that $(1-t^{2m})=(1+t^m)(1-t)(1+\dots + t^{m-1})$. Then by substituting that in and using algebraic manipulations we have 
\begin{align*}
   \frac{2R(t^m)}{1-t} &= \frac{1}{9}\sqrt{\frac{(1 - t^{2m})(80 - 36 t^{2}  
   -44 t^{2m}+ 
   24\sqrt{11} \sqrt{1 - t^{2m}} \sqrt{1 - t^{2}})}{(1-t)^2}} \\
   &= \frac{1}{9}\sqrt{(1+t^m)(1+t+\dots + t^{m-1})\left (\frac{80 - 36 t^{2}- 
   44 t^{2m}}{1-t} + 
   24\sqrt{11} \sqrt{(1+t^m)(1+\dots + t^{m-1})(1+t)}\right )}.
\end{align*}
Using standard techniques (e.g. L'Hopital's rule), one can compute the limit of $\frac{2R(t^m)}{1-t}$ and obtain 
\begin{equation} \label{eqn:glimit} g(m):=\lim_{t\rightarrow 1} \frac{2R(t^m)}{1-t}= \frac{1}{9}\sqrt{2m(88m+72 +48\sqrt{11m})}.
\end{equation}
Now we will find \eqref{lim:2}. Notice that we can simplify the expression to obtain $$\frac{1-c(t^m)^2}{1-t}+\frac{R(t^m)^2}{1-t}=\frac{1-c(t^m)^2}{1-t}+\frac{R(t^m)}{1-t}R(t^m).$$ Since $\lim_{t\rightarrow 1^-}R(t^m)=0$, it remains to find $\lim_{t\rightarrow 1^-}\frac{1-c^2(t^m)}{1-t}$. By substituting in for $1-t^{2m}$ and simplifying, we obtain
\begin{align*}
    \frac{1-c(t^m)^2}{1-t} 
     =&  \frac{1-\frac{23}{36}t^{2m}+\frac{11}{144}t^{4m}-\frac{3}{8}t^{m+1}-\frac{1}{16}t^2+\frac{11}{144}t^{2m+2}-\frac{11}{144}t^{4m+2}}{1-t}\\
     &+\left (\frac{1}{8}t^{2m}+\frac{1}{24}t^{m+1}\right )\sqrt{11}\sqrt{1+t}\sqrt{(1+t^m)(1+\dots + t^{m-1})}.
\end{align*}
Using standard techniques, one can compute the limit of $\frac{1-c(t^m)^2}{1-t}$ and obtain
\begin{equation} \label{eqn:hlimit}
    h(m):=\lim_{t\rightarrow 1^-} \frac{1-c(t^m)^2+R(t^m)^2}{1-t}=\lim_{t\rightarrow 1^-} \frac{1-c(t^m)^2}{1-t}= \frac{1}{2} (3 m+1)+\frac{\sqrt{11 m}}{3}.
\end{equation}
Lastly, we will find \eqref{lim:3}. Notice that this function is a combination of the other functions we already  considered and so, its limit must equal $\sqrt{-g(m)^2 +h(m)^2}$, where $g(m)$ and $h(m)$ were defined in \eqref{eqn:glimit} and \eqref{eqn:hlimit}. Inserting everything back into the original equation for the limit of $r(t)$ (which implicitly depends on $m$), we obtain the following limit function
\begin{equation} \label{eqn:ln} \ell(m):=\lim_{t\rightarrow 1^-} r(t^m)=\frac{9 + 6 \sqrt{11 m} + 27m - \sqrt{
 81 + 108\sqrt{11m} + 306m - 60m\sqrt{11m} + 
  25 m^2}}{8\sqrt{m \left (9 + 6\sqrt{11m} + 11m\right )}}.\end{equation}

\noindent
\textbf{Step 3: Showing that the limit of $r(t)$ is greater than $\frac{1}{\sqrt{2}}$}. By the reasonably simple formula for $r(t)$ and the fact that $\lim_{t\rightarrow 1^-} r(t)$ exists, we can define $r(t)$ so that it is left-continuous at $t=1$ and its value at $t=1$ is exactly the limit $\ell(m)$ from \eqref{eqn:ln}. Then, if $\ell(m)> \tfrac{1}{\sqrt{2}}$, there must be some interval $(t_m,1)$ such that $r(t) >\tfrac{1}{\sqrt{2}}$ for all $t \in (t_m,1)$.

Thus, to finish the proof, we just need to show that $\ell(m)> \tfrac{1}{\sqrt{2}}$ for each $m\in \mathbb{N}$ with $m \ge 2$. 
We will actually show that $\ell(x)>\tfrac{1}{\sqrt{2}}$ for all $x\in (1.7, \infty).$ 
Proceeding by contradiction, assume that there is an $x_*\in (1.7,\infty)$ such that $\ell(x_*)\le \frac{1}{\sqrt{2}}$. 
Since $\ell$ is continuous on $(1.7, \infty)$ and $\ell(2)>\tfrac{1}{\sqrt{2}}$, there must exist an $x' \in (1.7, \infty)$ such that $\ell(x') = \frac{1}{\sqrt{2}}$.
Then $\sqrt{x'}$ also satisfies $\ell(x^2) = \frac{1}{\sqrt{2}}$. 
Using algebraic manipulations similar to those in the proof of Theorem \ref{thm:n27}, one can show that
$\sqrt{x'}$ must be a zero of the polynomial
$$q(x)= 81 x^2 + 162 \sqrt{11} x^3 + 1125 x^4 + 180 \sqrt{11} x^5 - 569 x^6 - 174 \sqrt{11} x^7 - 77 x^8.$$ 
Using Mathematica, one can locate (up to a small error) all $8$ zeros of $q(x)$ and confirm that none of them are above $1.3.$ However, our assumption that $x' \in( 1.7, \infty)$ implies the $\sqrt{x'}$ satisfies both $\sqrt{x'} > 1.3$ and $q(\sqrt{x'})=0$, which gives our contradiction. Therefore, $\ell(x)>\tfrac{1}{\sqrt{2}}$ for all $x\in (1.7, \infty),$   which finishes the proof.
\end{proof}

\subsection{Other Blaschke products: the $4 \times 4$ case} \label{subsec:4x4} We now consider functions of the following form, with a repeated zero at $t$ and an additional zero at $\sqrt{t}$:
\begin{equation} \label{eqn:thetat3}  \Psi_t(z):= \left(\frac{z-t}{1-t z}\right)^3\left(\frac{z-\sqrt{t}}{1-\sqrt{t}z}\right).\end{equation}
Somewhat surprisingly, the pseudohyperbolic disk criteria holds for all $\Psi_t$ of form \eqref{eqn:thetat3}.

\begin{theorem} \label{thm:44} Let $t\in[0,1)$ and let $\Psi_t$ be a degree-four Blaschke product of the form \eqref{eqn:thetat3}. Then $W(M_{\Psi_t})$ contains a pseudohyperbolic disk with radius $(\frac{1}{2})^{\frac{1}{3}}$. 
\end{theorem}

\begin{proof}
Define the approximating curve $C_t$ by $\langle M_{\Psi_t} \vec{x}(s), \vec{x}(s)\rangle$ where $\vec{x}$ is from \eqref{eqn:vec1} with $n=4$. 
Then the points on $C_t$ are given by the function
\[
\begin{aligned}
    f(s) =& \frac{1}{20} \Big( \sqrt{5} t+15 t-\sqrt{5} \sqrt{t}+5 \sqrt{t}+ e^{i s} \left(-3 \sqrt{5} t^2-5 t^2-2 \sqrt{5} \sqrt{t+1} t+2 \sqrt{5} \sqrt{t+1}+3 \sqrt{5}+5\right)\\
    &+e^{2 i s} \left(2 \sqrt{5} t^3+2 \sqrt{5} \sqrt{t+1} t^2-2 \sqrt{5} t-2 \sqrt{5} \sqrt{t+1} t\right)\\
    &+ e^{3 i s} \left(\sqrt{5} \sqrt{t+1} t^3-5 \sqrt{t+1} t^3-\sqrt{5} \sqrt{t+1} t^2+5 \sqrt{t+1} t^2\right)\Big).
\end{aligned}
\]
As in Section \ref{sec:phd}, define a center $c(t)$ by
$$c(t)=\frac{f(\pi)+f(0)}{2}=\frac{1}{20} \left(2 \sqrt{5} t^3+2 \sqrt{5} \sqrt{t+1} t^2-\left(2 \sqrt{5} \sqrt{t+1}+\sqrt{5}-15\right) t-\left(\sqrt{5}-5\right) \sqrt{t}\right).$$ Setting $x=\cos(s)$ and simplifying, we have 
\begin{align}
      |f(s)-c(t)|^2 =R(t)^2+ \frac{1}{80} (t-1)^2 (t+1) g(t,x),
\end{align}
where 
\begin{align*}
   R(t)^2 =& -\frac{1}{40} (t-1)^2 (t+1) \Big(\sqrt{5} t^4-3 t^4-4 t^3-8 \sqrt{t+1} t^2-2 \left(-2 t+\sqrt{5} \sqrt{t+1}-3 \sqrt{t+1}+\sqrt{5}-5\right) t^2\\
   &-8 t^2-3 \sqrt{5} t-7 t-2 \sqrt{5} \sqrt{t+1}-6 \sqrt{t+1}-3 \sqrt{5}-\frac{17}{2}  \Big)\\
     g(t,x)=& 8t^2\left(1-x^2\right) \left( 5 -\sqrt{5}+2t+3\sqrt{1+t} -\sqrt{5}\sqrt{1+t}+2tx -2tx\sqrt{5}+2tx\sqrt{1+t} - 2\sqrt{5} tx \sqrt{1+t}\right) +1
\end{align*}
Using a standard (though tedious) calculus computation, one can show that $g(t,x)\geq 0$ on $[0,1]\times [-1,1]$. Following the arguments from Section \ref{sec:phd}, this implies that the Euclidean disk $D(c(t), R(t))\subseteq W(M_{\Psi_t}).$
This disk is also a pseudohyperbolic disk, whose pseuohyperbolic radius $r(t) \in [0,1)$ satisfies 
\begin{equation} \label{eqn:crt} r(t)+\frac{1}{r(t)}=\frac{R(t)^2-c(t)^2+1}{R(t)}.\end{equation}
To show $r(t)> \left(\frac{1}{2}\right)^{\frac{1}{3}}$ for $t\in (0,1)$, first observe that $r(t)$ will be continuous on $[0,1)$ and $r\left(\tfrac{1}{2}\right) > \left(\frac{1}{2}\right) ^{\frac{1}{3}}.$ If $r(t)<\left(\frac{1}{2}\right) ^{\frac{1}{3}}$ for some $t\in[0,1)$, there must be a $t'\in[0,1)$ where $r(t') = \left(\frac{1}{2}\right) ^{\frac{1}{3}}.$ Manipulating \eqref{eqn:crt} (in a way very similar to the proof of Theorem \ref{thm:n27}) will imply that $\sqrt{t'}$ is the zero of a degree $56$ polynomial $p$. However, one can use  Mathematica (or similar computer software) to locate the approximate zeros of $p$ and conclude that all of them lie outside of $[0,1)$. This gives the contradiction and  completes the proof. \end{proof}


\section{Crouzeix's Conjecture for Nilpotent Matrices} \label{sec:At}

In this section, we deepen the analysis from \cite{BG21} about Crouzeix's conjecture for certain nilpotent matrices and extend that analysis to new classes of nilpotent matrices. 

\subsection{Overview of Method} \label{sec:nil0}
We use curves approximating numerical ranges to study Crouzeix's conjecture both 
 for certain cases of the $n\times n$ matrices $A_t$ from \eqref{eqn:At} and for related classes of nilpotent matrices that involve a parameter $m$, which we denote $A_{t,m}$.  Here is the idea, which works equally well for the $A_t$ and $A_{t,m}$ matrices, though we use the $A_t$ notation below for simplicity:\\

\begin{itemize}
\item[1.] Let $\vec{x}(s)$ be the vector-valued function defined in \eqref{eqn:vec1} and let $C_t$ denote the curve of points in $W(A_t)$ given by $f(s):=\left \langle A_t \vec{x}(s), \vec{x}(s) \right \rangle$ for $s\in[0, 2\pi].$\\

\item[2.] Define a function $F(z)$ by $F(e^{is}) = f(s).$ In all of our situations, $F$ is a polynomial (whose coefficients might depend on $t$) with $F(0)=0$ and $F'(0) \ne 0.$ Since $F(\mathbb{T}) = C_t$, properties of holomorphic functions imply that $F(\mathbb{D})$ is contained in the convex hull of $C_t$, which in turn is in $W(A_t)$, see Remark $5.7$ in \cite{BG21}. \\

\item[3.] Find an $n \times n $ matrix $B_t$ such that $F(B_t) = A_t.$ Specifically, the fact that $F'(0) \ne 0$ implies that $F$ is locally invertible near $0$. Let $F^{-1}$ denote this local inverse and let $G_n$ denote the degree $n-1$ Taylor polynomial of  $F^{-1}$ centered at $0$. Set $B_t = G_n(A_t)$ and note that $B_t$ also equals $F^{-1}(A_t)$, since $A_t$ is nilpotent with order at most $n$. In the situations we consider, the Jordan decomposition of $B_t$ is always of the form $X_t J_n X_t^{-1}$, where $J_n$ is the standard $n\times n$ Jordan block  and $X_t$ is an $n\times n$ invertible matrix.\\

\item[4.] Use this to analyze Crouzeix's conjecture for $A_t$ by first fixing any $p\in \mathbb{C}[z]$.
Then, as in \cite{BG21},  we have the following sequence:
 \begin{align}
 \|p(A_t)\| &=  \|(p \circ F)(B_t)\| \nonumber \\
&= \|X_t (p \circ F)(J_n) X_t^{-1}\| &\notag\\
 & \le   \|X_t\|  \|X_t^{-1}\| \cdot \|(p \circ F)(J_n)\| \nonumber \\
& \le \|X_t\|  \|X_t^{-1}\| \sup_{z \in \mathbb{D}} |(p \circ F)(z)| \nonumber \\
&  \le \|X_t\|  \|X_t^{-1}\| \sup_{z \in W(A_t)} |p(z)|, \label{eqn:cc}
 \end{align} 
 where we used von Neumann's inequality applied to the contraction $J_n$ and the fact that $F(\mathbb{D}) \subseteq W(A_t).$ Then, $A_t$ will satisfy Crouzeix's conjecture as long as $\|X_t\|  \|X_t^{-1}\| \le 2.$ 
\end{itemize}

Throughout Section \ref{sec:At}, we keep this construction at the forefront and primarily restrict to studying the properties of $\|X_t\|$ and $ \|X_t^{-1}\|$. It is worth noting that  the Jordan decomposition of $B_t$ is not unique and there are multiple matrices $X_t$ for which $X_tJ_nX_t^{-1}=B_t$. We could use any of these $X_t$ to bound $||p(A_t)||$. However, we checked several situations and the $X_t$ we use (specifically, the ones provided by Mathematica's Jordan Decomposition command) appear to generally give the lowest value for $||X_t||  \ ||X_t^{-1}||$.

\subsection{Nilpotent matrices from $M_{\Theta}$ matrices} \label{subsec:nil1}
We first study the nilpotent matrices $A_t$ from \eqref{eqn:At} that arose naturally in the study of $M_{\Theta_t}$ in the $n=4$ case in \cite{BG21}. Following the arguments in Subsection \ref{sec:nil0}, we obtain 
  \begin{equation}  X_t = \begin{pmatrix} 1 & 0 & 0 & 0 \\
   0 & \tfrac{1}{4} (1 + \sqrt{5}) & -\tfrac{3}{40} (-5 + \sqrt{5}) t & -\tfrac{t^2}{
  8 \sqrt{5}} \\[.2 em]
    0 & 0 & \tfrac{1}{8} (3 + \sqrt{5}) & \tfrac{3 t}{4 \sqrt{5}} \\[.2 em]
     0 & 0 & 0 & \tfrac{1}{8} (2 + \sqrt{5}) \end{pmatrix}.\label{eqn:Xt4} \end{equation}

     The details, including formulas for $F$ and $B_t$, also appear in Remark 5.8 in \cite{BG21}. To analytically study these $X_t$ matrices and in particular establish Crouzeix's conjecture  for certain $A_t$ matrices via \eqref{eqn:cc}, we will need to understand the zeros of certain cubic polynomials. The needed information is encoded in the following remark.

\begin{remark} \label{remark: Depressed Cubic and Orderings}
 Consider a cubic polynomial given by 
 \begin{equation} \label{eqn:R} R(x) = ax^3+bx^2+cx+d,\end{equation}
 for $a, b, c, d \in \mathbb{R}$ and let $x_0$, $x_1$,  $x_2$ denote the zeros of $R(x)$. Further, assume that we know that $x_0, x_1, x_2 \in \mathbb{R}$ and they are not all the same. To find the formulas for these zeros, we first convert $R$ into a depressed cubic $Q$ via the formula $Q(x)= R\left(x-\frac{b}{3a}\right).$ This yields
 \begin{align*}Q(x)&=x^{3}+px+q, \ \ \text{ where } \ \ p=\frac{3 a c-b^2}{3 a^2} \text{ and } q=\frac{2b^3-9abc+27a^2 d}{27 a^3}.\end{align*}
Because $R$ has real zeros that are not all the same, we can assume that $p \ne 0$. If we additionally know that
 \begin{equation} \label{eqn:G}  -1 \le G \le 1, \ \text{ where } \ \ G:=\frac{3 q\sqrt{-\frac{3}{p}} }{2 p}, \end{equation}
 then the zeros $z_0, z_1, z_2$ of $Q$ are given by
\begin{equation}z_k = 2 \sqrt{-\frac{p}{3}} \cos \left(\frac{1}{3} \cos ^{-1}\left(\frac{3 q\sqrt{-\frac{3}{p}} }{2 p}\right)-\frac{2 \pi  k}{3}\right) \text{ for $k=0,1,2$.}\label{eqn: wk}\end{equation}
These formulas are well known and appear for example in \cite{Nick93}. Then the zeros of $R(x)$ are given by 
\begin{equation} \label{eqn:xk} x_k = z_k-\frac{b}{3a}, \ \ \text{ for }  \ \ k=0,1,2.\end{equation}
 It is also true that $x_2 \le x_1 \le x_0.$ To see this, observe that standard trigonometric identities imply that if $s\in[0,\pi],$ then  $$  \cos \left(\frac{s}{3}\right) \ge \cos \left(\frac{s}{3} -\frac{2 \pi  }{3}\right) \ge \cos \left(\frac{s}{3} -\frac{4 \pi  }{3}\right).$$  Applying this with s = $\cos ^{-1}\left(\frac{3q\sqrt{-\frac{3}{p}} }{2 p}\right)$ gives the desired ordering. 
\end{remark}

We now use the ideas in Remark \ref{remark: Depressed Cubic and Orderings} to study the norm of the matrix $X_t$.

\begin{theorem} \label{thm:norm1} For $t \in [0,1)$,
the matrix $X_t$ from \eqref{eqn:Xt4} satisfies $\|X_t\|=1.$ 
\end{theorem}
\begin{proof}
Since $\| X_t\| = \text{max} \{\sqrt{\lambda}: \lambda \in \sigma(X_t^*  X_t)\}$, we need to find $\sigma(X_t^*  X_t)$. To that end, consider the polynomial $
     S(x) = \det\left(X_t^*  X_t - xI\right),$ 
    whose zero set is exactly $\sigma(X_t^* X_t)$. One can check that $S$ factors as $ S(x) = \frac{1}{102400} R(x)(x-1)$, where
\begin{align*}R(x)&=102400 x^3+\left(-320 t^4+5760 \sqrt{5} t^2-28800 t^2-28800 \sqrt{5}-75200\right) x^2 \\&+\left(-546 \sqrt{5} t^4+2374 t^4+1710 \sqrt{5} t^2+4950 t^2+13350 \sqrt{5}+29950\right) x-1800 \sqrt{5}-4025.\end{align*}
Then $\sigma(X_t^*  X_t) = \{ 1, x_0, x_1, x_2\}$, where $x_0, x_1, x_2$ are the (necessarily real and non-negative) zeros of $R$. Looking at the structure of $X_t^*  X_t$, we can further assume that those zeros are not all the same. 

Denote the coefficients of $R$ using $a,b,c,d$ as in \eqref{eqn:R} and define $p$, $q$, and $G$ (functions of $t$) as in Remark \ref{remark: Depressed Cubic and Orderings}. Then $p(t) \ne 0$ in $[0,1].$ To establish \eqref{eqn:G}, first rearrange the terms in the definition of $G$ to conclude that 
\[
G(t) =-\frac{q}{2} \left( \frac{-p}{3} \right)^{-3/2}. \]Taking derivatives and using the fact that $p \ne 0$ on $[0,1]$ implies that 
\[ G'(t) = 0 \text{ on } [0,1] \ \ \ \text{ if and only if }  \ \ \ -3 q(t) p'(t) + 2q'(t)p(t)=0.\]
However, the second equation above is a polynomial equation of degree $11$ in $t$. One can easily use computer software like Mathematica to locate (within some small error) all $11$ of its solutions and conclude that the only one in $[0,1]$ is $t=0$. Since $G'(1/2) <0$, this shows that $G(t)$ is decreasing on $[0,1]$. Thus for $t\in[0,1],$ 
\begin{equation} \label{eqn: Grange}
     -0.542 < G(1) \le G(t) \le G(0) < 0.353.\end{equation}
This establishes \eqref{eqn:G} and so, we can conclude that the largest zero of $R(x)$ is
\begin{equation}
    x_0(t)= 2 \sqrt{-\frac{p}{3}} \cos \left(\frac{1}{3} \cos ^{-1}\left(\frac{3 q\sqrt{-\frac{3}{p}} }{2 p}\right)\right)-\frac{b}{3 a}. \label{eqn:x_k} 
\end{equation}
To complete the proof, we just need to show that $|x_0| \le 1.$ Because $|\cos(x)| \le 1$, we just need to prove that $ |2\sqrt{ \frac{-p}{3}}| + |\frac{b}{3a}| \le 1.$ By removing the only negative coefficient and using $t \in[0,1]$, we have
\begin{align*}
   \left |2\sqrt{ \frac{-p}{3}}\right| &= \frac{1}{480} \sqrt{t^8+36 \left(5-\sqrt{5}\right) t^6-2 \left(711 \sqrt{5}-1534\right) t^4+90 \left(29 \sqrt{5}+125\right) t^2+125 \left(18 \sqrt{5}+47\right)} \\
  &\le \frac{1}{480} \sqrt{1+ 36 \left(5-\sqrt{5}\right) +90 \left(29 \sqrt{5}+125\right)+125 \left(18 \sqrt{5}+47\right)} <  0.35.
 \end{align*}
Meanwhile, the $|\frac{b}{3a}|$ term is a polynomial with positive coefficients. So for $t  \in[0,1]$,
\[
\left |\frac{b}{3a}(t)\right| \le \left |\frac{b}{3a}(1)\right| <0.51. \]
Combining the two estimates shows that $|x_0| \le 1$, which is what we needed to prove.   \end{proof}

We use similar techniques to study $\|X_t^{-1}\|.$

\begin{theorem} \label{thm:XtI}
The matrix $X_t$ from \eqref{eqn:Xt4} satisfies $\|X_t^{-1}\| \le 2.37$ when $t \in [0,1)$ and  $ \|X_t^{-1}\| \le 2$ when $ t \in [0, 0.363].$
\end{theorem}

\begin{proof}
Using standard properties of norms and the ordering given in Remark \ref{remark: Depressed Cubic and Orderings}, we have
\[
\|X_t^{-1}\| = \max\left\{ 1, \frac{1}{\sqrt{x_0}}, \frac{1}{\sqrt{x_1}}, \frac{1}{\sqrt{x_2}} \right\} = \frac{1}{\sqrt{x_2}}.
\] 
To find an upper bound for $\|X_{t}^{-1}\|$, we need to find a lower bound for $x_2(t)$ on $[0,1].$ Using the notation in Remark \ref{remark: Depressed Cubic and Orderings}, we can write \(x_2 = z_2 - \frac{b}{3a}\). Since \(- \frac{b}{3a}\) is a positive and increasing polynomial on \([0,1]\), it will be bounded below by its value at $t=0.$ Thus, we can focus on \(z_2\), which has formula
\[
z_2(t) = 2 \sqrt{\frac{-p}{3}} \cos \left( \frac{1}{3} \cos^{-1}(G(t)) - \frac{4\pi}{3} \right),
\]
where \(G(t)\) is from \eqref{eqn:G}. Using the formula for $2 \sqrt{\frac{-p}{3}}$ from the proof of Theorem \ref{thm:norm1}, one can easily see that it is continuous, positive, and increasing on \([0,1]\). Furthermore, the proof of Theorem \ref{thm:norm1} implies that $G$ is decreasing on $[0,1]$ and $G([0,1]) \subseteq [-0.542, 0.353]$. Then,  \(\cos^{-1}(G([0,1]))\subseteq [0,\pi]\) and for $t \in[0,1]$, 
\[ \cos^{-1}\left(G(t)\right) - \frac{4 \pi }{3} \in [\tfrac{-4\pi}{3}, \tfrac{-\pi}{3}].\] 
Fix $t^* \in [0,1]$. Then we can draw the following conclusions:
\begin{itemize}
\item \(G(t)\) attains its minimum on $[0,t^*]$ at \(t = t^*\). 
\item Since $\cos^{-1}(x)$ is a decreasing function, \(\cos^{-1}(G(t))\) attains its maximum on $[0, t^*]$ at  $t = t^*$. 
\item Since $\cos(x)$ is decreasing on $[\tfrac{-4\pi}{3}, \tfrac{-\pi}{3}]$, the cosine term in $z_2$ attains its minimum on $[0,t^*]$ at $t=t^*$. 
\item Since $f$ is increasing and positive and the cosine term is negative on $[0,t^*]$, $z_2$ attains its minimum on $[0,t^*]$ at $t=t^*$.
\end{itemize}
Thus, we have
\[
\frac{1}{\|X_t^{-1}\|} = \sqrt{x_2} = \sqrt{z_2(t) - \frac{b}{3a}(t)} \ge \sqrt{z_2(t^*) - \frac{b}{3a}(0)}.
\]
Selecting \( t^* = 0.363 \), we find
\[
\|X_t^{-1}\| \le \frac{1}{\sqrt{z_2(0.363) - \frac{b}{3a}(0)}} < 1.9999,
\] for \( t \in [0, 0.363] \). Similarly, by setting \( t^* = 1 \), it follows that
\[
\|X_t^{-1}\| \le \frac{1}{\sqrt{z_2(1) - \frac{b}{3a}(0)}} \le 2.83,
\]
for \( t \in [0, 1) \), which completes the proof.
\end{proof}

Theorems \ref{thm:norm1} and \ref{thm:XtI} imply the immediate corollary:

\begin{corollary} \label{cor:44}
Let $n=4$ and $A_t$ be given in \eqref{eqn:At}. Then $A_t$ satisfies Crouzeix's conjecture for $t \in [0,0.363].$ 
\end{corollary}

As discussed in Remark $5.8$ in \cite{BG21}, numerical work indicates that Corollary \ref{cor:44} should actually hold for all $t\in[0,0.42]$, though proving that analytically seems challenging. Remark $5.8$ in \cite{BG21} also numerically explores $\| X_t\|  \| X_t^{-1} \|$ for $A_t$ as in \eqref{eqn:At} with $n=5$. When $n \ge 6$, there are no longer simple formulas for $X_t$. However, the $X_t$ matrices can still be computed and studied numerically for  larger $n$ values. We investigated this and record our findings in the following remark.

\begin{remark} \label{rem:68}
For $A_t$ defined via \eqref{eqn:At} with \( n = 6, 7, 8 \), we found matrices $X_t$ using the process outlined in Subsection \ref{sec:nil0} and investigated the associated product $\|X_t\| \ \|X_t^{-1}\|$ for $t\in[0,1]$.
The complexity of the matrices meant that they were not amenable to the 
Maximize function in Mathematica. Instead, we graphed \( \|X_t\|  \|X_t^{-1}\| \) as functions of \( t \). The plots suggest that these functions are increasing, so the maximum value likely occurs at $t=1$. We also found likely intervals where \( \|X_t\|  \|X_t^{-1}\| \) remains below $2$ by finding $t$-values  where the product is below, but very close to, this threshold. Figure \ref{fig:Xtplot} gives the plots of \( \|X_t\|  \|X_t^{-1}\| \) for \( n = 6, 7, 8 \). Table \ref{table:Xt1} summarizes these computations, with the \( t \)-values approximated to the nearest hundredth by rounding down. This computational study suggests the following: for the $n \times n$ matrix $A_t$, Crouzeix's conjecture holds for $t \in [0,0.545]$ when $n=6,$ for $t \in [0,0.580]$ when $n=7,$ and for $t \in [0,0.608]$ when $n=8.$
\end{remark}
\begin{figure}[h!]
  \centering

  \begin{minipage}[b]{0.4\textwidth}
    \centering
    \small 
    \renewcommand*{\arraystretch}{1.2}
    \begin{tabular}{||c|c|c||} 
      \hline
      \( n \) &  $\|X_{t}\| \|X_{t} ^{-1}\|$ Bound  & \( t \)-interval w/ product \( < 2 \) \\
      \hline
      6 & 2.63 & [0, 0.545] \\ 
      \hline
      7 & 2.73 & [0, 0.580] \\  
      \hline 
      8 & 2.81 & [0, 0.608] \\ 
      \hline
    \end{tabular}
    
    \caption{Summary of numerical investigations for $A_{t}$ and $X_t$ with larger $n$ values.}
    \label{table:Xt1}
  \end{minipage} 
  \qquad     \qquad   
  \begin{minipage}[b]{0.5\textwidth}
    \centering
    \includegraphics[width=2.5in]{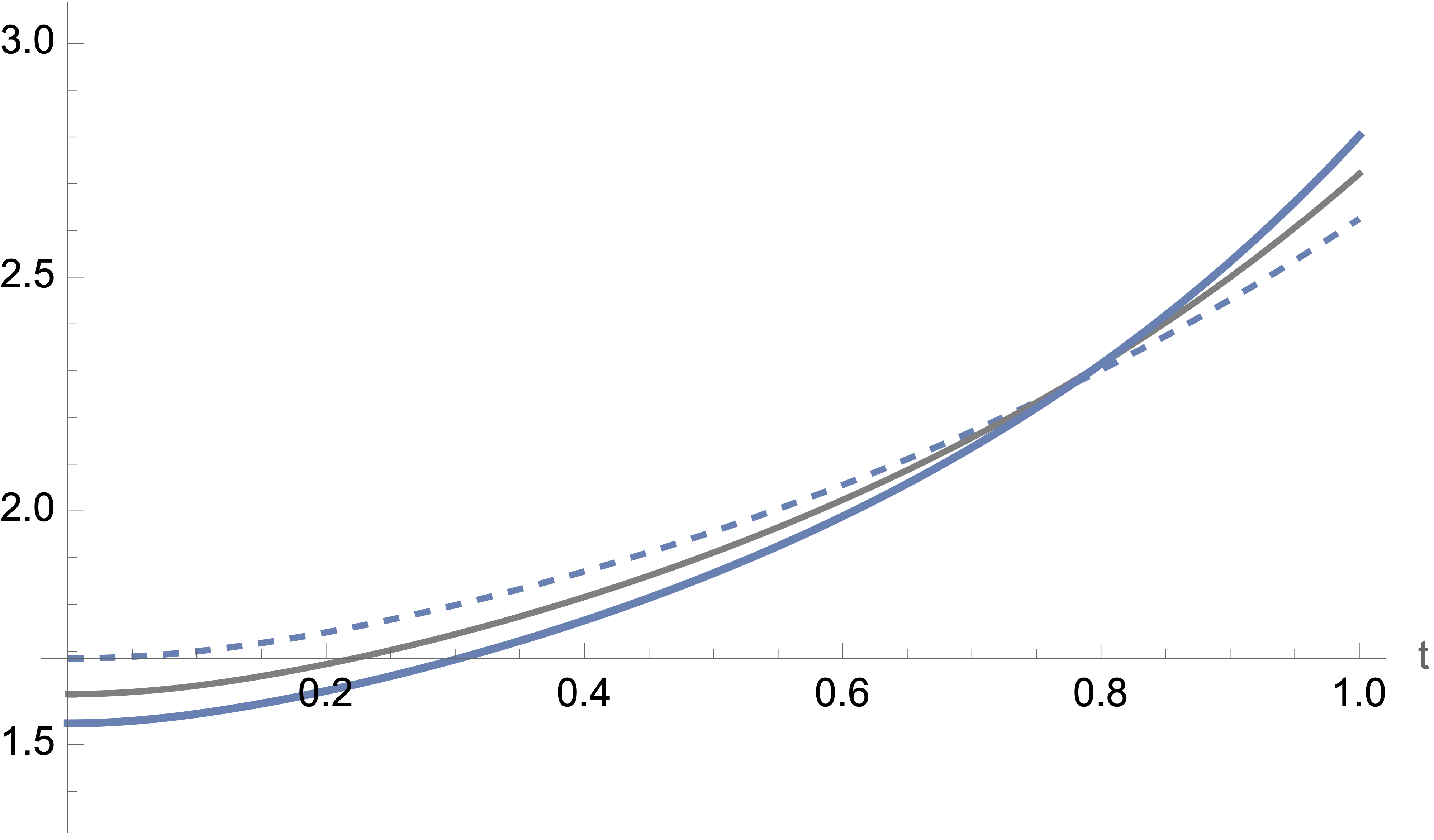} 
    \caption{Plots of \( \|X_t\| \|X_t^{-1}\| \) for \( n=6 \) (dashed), \( n=7 \) (gray), and \( n=8 \) (black).}
    \label{fig:Xtplot}
  \end{minipage}
\end{figure}


\subsection{A new class of nilpotent matrices}  \label{subsec:nil2}
In this section, we study a related class of nilpotent matrices $A_{t,m}$ given as follows 
$$A_{t,m} = \left(
\begin{array}{cccc}
 0 & 1 & t & t^m \\
 0 & 0 & 1 & t \\
 0 & 0 & 0 & 1 \\
 0 & 0 & 0 & 0 \\
\end{array} \label{eqn:newAt}
\right), \text{ where } m \ge 2.$$
Following the steps in Subsection \ref{sec:nil0}, let $\vec{x}$ be given by \eqref{eqn:vec1} and let $C_t$ denote the curve given by $\left \langle A_{t,m} \vec{x}(s), \vec{x}(s) \right \rangle$ for $s\in[0, 2\pi].$ This allows us to define $F$, which has the formula
 \[ F(z) = \frac{1}{20} \left(-\left(\left(\sqrt{5}-5\right) z^{3} t^m\right)+4 \sqrt{5} z^{2} t+5 \left(\sqrt{5}+1\right) z \right). \] 
 With $F$ in hand, we can find matrices $B_{t,m}$ and $X_{t,m}$ such that  $F(B_{t,m}) = A_{t,m}$ and $B_{t,m} = X_{t,m} J_4 X_{t,m}^{-1}$. The explicit formulas for $B_{t,m}$ and $X_{t,m}$ are
\[ B_{t,m} = \left(
\begin{array}{cccc}
 0 & \sqrt{5}-1 & \left(\frac{21}{\sqrt{5}}-9\right) t & \frac{32 \left(7 \left(\sqrt{5}+5\right) t^m-24 \left(\sqrt{5}-1\right) t^2\right)}{5 \left(\sqrt{5}+1\right)^5} \\
 0 & 0 & \sqrt{5}-1 & \left(\frac{21}{\sqrt{5}}-9\right) t \\
 0 & 0 & 0 & \sqrt{5}-1 \\
 0 & 0 & 0 & 0 \\
\end{array}
\right)  \]and 
  \begin{equation}  X_{t,m} = \left(
\begin{array}{cccc}
 1 & 0 & 0 & 0 \\
 0 & \frac{1}{4} \left(\sqrt{5}+1\right) & \frac{3}{40} \left(\sqrt{5}-5\right) t & \frac{6 t^2-7 t^m}{8 \sqrt{5}} \\
 0 & 0 & \frac{1}{8} \left(\sqrt{5}+3\right) & -\frac{3 t}{4 \sqrt{5}} \\
 0 & 0 & 0 & \frac{1}{8} \left(\sqrt{5}+2\right) \\
\end{array}
\right).\label{eqn:Xt4fornewfamily} \end{equation}
Using arguments similar to those in the previous subsection, we will explore Crouzeix's conjecture for these matrices $A_{t,m}$. As described in Subsection \ref{sec:nil0}, we need to understand the norms $\| X_{t,m}\|$ and $\| X_{t,m}^{-1} \|$. We first have the following result.
\begin{theorem} \label{thm:Xt2}
For \( m \in \{2, 3, 4\} \), the matrix \( X_{t,m} \) from \eqref{eqn:Xt4fornewfamily} satisfies \( \| X_{t,m} \| \leq 1 \) for all \( t \in [0,1) \). 
\end{theorem}
\begin{proof}
As in the proof of Theorem \ref{thm:norm1}, we can consider the polynomial
\[
S(x) = \det\left(X_{t,m}^* X_{t,m} - xI\right),
\]
since the square root of its largest zero is exactly $\| X_{t,m}\|$. This polynomial factors as $S(x) = \frac{1}{102400} R(x)(x-1)$, where
\begin{align*}
R(x) =\ & x^2 \left(-15680 t^{2 m} + 26880 t^{m+2} - 11520 t^4 + 5760 \sqrt{5} t^2 - 28800 t^2 - 28800 \sqrt{5} - 75200\right) \\
        & + x \left(1470 \sqrt{5} t^{2 m} + 3430 t^{2 m} - 2016 \sqrt{5} t^{m+2} - 3360 t^{m+2} + 2304 t^4 + 1710 \sqrt{5} t^2 \right. \\
        & \left. \quad + 4950 t^2 + 13350 \sqrt{5} + 29950\right) \\
        & + 102400 x^3 - 1800 \sqrt{5} - 4025.
\end{align*}
To be consistent with \eqref{eqn:R}, we let $a,b,c,d$ denote the coefficients of $R$ and let $p_m$, $q_m$, and $G_{m}$ be the associated functions (depending on $t$ and $m$) that are defined Remark \ref{remark: Depressed Cubic and Orderings}. Looking at $X_{t,m}$, one can conclude that $R$ does not have a single repeated zero and so $p_{m}(t) \ne 0$ on $[0,1].$ To obtain the needed inequality in \eqref{eqn:G}, we first observe that
\[G_{m}(t) = -\frac{q_m}{2} \left( \frac{-p_m}{3} \right)^{-3/2}. \label{eqn:Gmt}\] 
Now we restrict to the $m=2$ case. Differentiating $G_2$ and using $p_{2} \ne 0$ on $[0,1]$, we can conclude that
\[ G_{2}'(t) = 0 \text{ on } [0,1] \ \ \ \text{ if and only if } \ \ \ 3 q_{2}(t) p_{2}'(t) - 2q_{2}'(t)p_{2}(t)=0.\]
The latter is a polynomial equation of degree 11 in $t$. Using Mathematica, one can pinpoint (within a small marginal error) all $11$ solutions and deduce that the only solution in $[0,1]$ is $t=0$. Given that $G_{2}'(1/2) <0$, this implies that $G_{2}(t)$ is decreasing on $[0,1]$ and so
\begin{equation} \label{eqn: Grange2}
-0.542 < G_{2}(1) \le G_{2}(t) \le G_{2}(0) < 0.353.\end{equation}
Repeating this method for $m=3$ and $m=4$ yields polynomial equations of degrees $17$ and $23$ respectively whose only solution in $[0,1]$ is $t=0$.  Since $G_3'(1/2)$ and $G_4'(1/2)$ are both negative, this implies that $G_3$ and $G_4$ are also decreasing on $[0,1]$. They actually have exactly the same upper and lower bounds as in \eqref{eqn: Grange2}.
Combined with Remark \ref{remark: Depressed Cubic and Orderings}, this implies that for $m \in \{2,3,4\}$, the largest zero of $R$ is
\[
    x_0(t)= 2 \sqrt{-\frac{p_{m}}{3}} \cos \left(\frac{1}{3} \cos^{-1}{ \left( G_{m}(t) \right)}\right)-\frac{b}{3 a}.\]
To finish the proof, we will show that \( |x_0| \leq 1 \). Since \( |\cos(x)| \leq 1 \), we just need \( |2\sqrt{-\frac{p_m}{3}}| + \left|\frac{b}{3a}\right| \leq 1 \). Here are the general formulas:
\begin{align*}
\left| 2\sqrt{ \frac{-p_m}{3} } \right| &= \frac{1}{480} \left( \left(49 t^{2m} - 84 t^{m+2} + 36 t^4 - 18 \left(\sqrt{5}-5\right) t^2 + 90 \sqrt{5} + 235\right)^2 \right. \\
&\quad \left. - 6 \left( 245 \left(3 \sqrt{5}+7\right) t^{2m} - 336 \left(3 \sqrt{5}+5\right) t^{m+2} + 1152 t^4 \right. \right. \\
&\quad \left. \left. + 45 \left(19 \sqrt{5}+55\right) t^2 + 25 \left(267 \sqrt{5}+599\right) \right) \right)^{\frac{1}{2}}
\end{align*}
and 
\[ \left |\frac{b}{3a}\right| = \frac{1}{960} \left(49 t^{2 m}-84 t^{m+2}+36 t^4-18 \left(\sqrt{5}-5\right) t^2+90 \sqrt{5}+235\right). \] 
We first consider $m=2.$ After simplifying, we can remove any negative coefficients and use $t \in [0,1]$ to conclude
 \begin{align*}
\left| 2\sqrt{ \frac{-p_2}{3} } \right| &=\frac{1}{480} \sqrt{t^8-36 \left(\sqrt{5}-5\right) t^6+\left(3068-1422 \sqrt{5}\right) t^4+90 \left(29 \sqrt{5}+125\right) t^2+125 \left(18 \sqrt{5}+47\right)} \\
&\le \frac{1}{480} \sqrt{t^8-36 \left(\sqrt{5}-5\right) t^6+90 \left(29 \sqrt{5}+125\right) t^2+125 \left(18 \sqrt{5}+47\right)} \le 0.35.
\end{align*}
For the $|\frac{b}{3a}|$ term when $m=2$, all of the coefficients are positive after simplifying, so \[ |\frac{b}{3a} (t)| \le|\frac{b}{3a} (1)| \le  0.51. \] Combining these estimates verifies that $|x_0| \le 1$. A similar analysis works for $m=3$ and $m=4$ and so, we omit the details here.
\end{proof}
We now use similar techniques to study $\|X_{t,m}^{-1}\|.$
\begin{theorem}\label{thm:newXtInorm}
 For $m \in \{2,3,4\}$ the matrix $X_{t,m}$ defined in \eqref{eqn:Xt4fornewfamily} satisfies $\| X_{t,m}^{-1}\| \le K_m$ for $t\in[0,1]$ and $\| X_{t,m}^{-1} \| \le 2$ for $t\in [0, t_m^*]$, where $K_m$ and $t_m^*$ are given in the following table:
 \[ 
 \renewcommand*{\arraystretch}{1.2}
\begin{tabular}{||c | c|c||} 
 \hline
 \hspace{0.1cm} $m$  \hspace{0.1cm} & \hspace{0.25cm} $K_m$\hspace{0.25cm} &  \hspace{0.25cm} $t^*_m$ \hspace{0.25cm} \\ [0.5ex]
 \hline
  2 & 2.83 & 0.363\\ 
 \hline
 3 & 2.83&  0.368\\ 
 \hline
 4 & 2.83& 0.367\\ [0.5ex]
 \hline
\end{tabular}
\]
\end{theorem} 
\begin{proof} In Theorem \ref{thm:Xt2}, we considered
\[
S(x) = \det\left(X_{t,m}^* X_{t,m} - xI\right) =\frac{1}{102400} R(x)(x-1),
\]
for a polynomial $R(x).$ Then \( \|X_{t,m}^{-1}\| = \frac{1}{\sqrt{x_2}}, \) where $x_2$ is the smallest zero of $R(x)$. Using  Remark~\ref{remark: Depressed Cubic and Orderings}, \(x_2\) has form
\begin{equation}
x_2(t) = f_m(t) \cos \left(\frac{1}{3} \cos^{-1}\left(G_{m}(t)\right) - \frac{4 \pi }{3}\right) - \frac{b}{3a},
\end{equation}
where \(G_m(t)\) is from \eqref{eqn:Gmt} and $f_m(t) = 2 \sqrt{\frac{-p_m}{3}}$ are both defined using the coefficients of $R(x).$ To find an upper bound for $\|X_{t,m}^{-1}\|$, we need to find a lower bound for $x_2(t)$ on $[0,1]$.

For each \(m \in \{2, 3, 4\}\), the term \(\frac{-b}{3a}\) is continuous, positive, and increasing. Thus it is minimized at $t=0$ and we can focus on \(H_m(t):=x_2 +\frac{b}{3a}\), which is the component of $x_2$ excluding \(\frac{-b}{3a}\), and analyze its behavior. One can check the coefficients of each $p_m$
to conclude that each $f_m$ is increasing on $[0,1]$. By the proof of Theorem \ref{thm:Xt2}, each $G_m$ is decreasing on $[0,1]$, and \(G_m([0,1])\subseteq [-0.542, 0.353]\). This implies that  \(\cos^{-1}(G_m([0,1]))\subseteq [0,\pi]\) and thus for $t \in[0,1]$, 
\[ \cos^{-1}\left(G_{m}(t)\right) - \tfrac{4 \pi }{3} \in [\tfrac{-4\pi}{3}, \tfrac{-\pi}{3}].\]
Fix $t^*\in[0,1]$. Then we can conclude that
\begin{enumerate}
\item \(G_m(t)\) attains its minimum on $[0,t^*]$ at \(t = t^*\). 
\item Since $\cos^{-1}(x)$ is a decreasing function, \(\cos^{-1}(G_m(t))\) attains its maximum on $[0, t^*]$ at  $t = t^*$. 
\item Since $\cos(x)$ is decreasing on $[\tfrac{-4\pi}{3}, \frac{-\pi}{3}]$, the cosine term in $H_m(t)$ attains its minimum on $[0,t^*]$ at $t=t^*$. 
\item Since $f_m$ is increasing and positive and the cosine term is negative on $[0,t^*]$, $H_m(t)$ attains its minimum on $[0,t^*]$ at $t=t^*$.
\end{enumerate}
Therefore, for \(m = 2, 3, 4\), we have
\[
\frac{1}{\|X_{t,m}^{-1}\|} = \sqrt{x_2} = \sqrt{H_m(t) - \frac{b}{3a}(t)} \ge \sqrt{H_m(t^*) - \frac{b}{3a}(0)}.
\]
By selecting \(t^* = 0.363\) for \(m = 2\), we obtain
\[
\|X_{t,m}^{-1}\| \le \frac{1}{\sqrt{H_m(0.363) - \frac{b}{3a}(0)}} < 1.9999
\]
for \(t \in [0, 0.363]\). Similarly, setting \(t^* = 1\), we find
\[
\|X_{t,m}^{-1}\| \le \frac{1}{\sqrt{H_m(1) - \frac{b}{3a}(0)}} \le 2.83
\]
for \(t \in [0, 1]\). Analogous estimates give the values in the table for $m=3$ and $m=4,$ which completes the proof.
\end{proof}

We have the following immediate corollary.

\begin{corollary} \label{cor:new44}
Let $A_{t,m}$ be given in \eqref{eqn:newAt}. For $m \in \{2,3,4\},$ $A_{t,m}$ satisfies Crouzeix's conjecture for $t \in [0,t_m^*],$ where $t_m^*$ is from Theorem \ref{thm:newXtInorm}.
\end{corollary}

\begin{proof} This is an immediate consequence of \eqref{eqn:cc}, Theorem \ref{thm:Xt2}, and Theorem \ref{thm:newXtInorm}.
\end{proof}

One can also explore $A_{t,m}$ matrices for higher values of $m$ and different dimensions. Our numerical explorations in this direction are discussed in the following remark.

\begin{remark} In Theorem \ref{thm:newXtInorm}, we restricted to $m\in\{2,3,4\}$ because, for higher  values of $m$, some of the key functions used in the proof are no longer monotonic. Thus, the arguments employed in 
Theorem \ref{thm:newXtInorm} fail for those values of $m$ and more complicated arguments would be required to obtain an analytic proof of an upper bound for $\| X^{-1}_{t,m}\|$.

Nonetheless, one can use computational methods to explore versions of Theorem \ref{thm:newXtInorm} and Corollary \ref{cor:new44} for additional values of $m$. One can also explore similar $A_{t,m}$ of different sizes $n$ beyond the $n=4$ case. For example, the $5\times 5$ $A_{t,m}$ has formula given below
\[A_{t,m} =\left( \begin{array}{ccccc}
 0 & 1 & t & t & t^m \\
 0 & 0 & 1 & t & t \\
 0 & 0 & 0 & 1 & t \\
 0 & 0 & 0 & 0 & 1 \\
 0 & 0 & 0 & 0 & 0 \\
\end{array}\right) \]
and general $n \times n$ $A_{t,m}$ can be similarly defined.
For such $A_{t,m}$, one can follow the process outlined in Subsection \ref{sec:nil0}, obtain a related matrix $X_{t,m}$, and investigate Crouzeix's conjecture.
We explored these additional cases using the Maximize command in Mathematica for the $n=4$ case, and by graphical methods for the $n=5$ and $n=6$ cases.  Figure \ref{table:1} contains a summary of our numerical investigations. Note that in the setting of different $m$-values but a fixed $n$-value, the table appears to shows that $\|X_{t,m} \|  \|X_{t,m} ^{-1}\|<2$ (and hence, $A_{t,m}$ satisfies Crouzeix's conjecture) on the same $t$-interval. These intervals are actually different in practice, but they appear to be the same here due to rounding down.

 \begin{figure}[h!]
\renewcommand*{\arraystretch}{1.2}
\begin{tabular}{||c |c |c| c||} 
 \hline
 $n$-value  & $m$-value & Upper Bound for $\|X_{t,m}\| \dot \|X_{t,m} ^{-1}\|$ & $t$ interval with $\|X_{t,m}\| \dot \|X_{t,m} ^{-1}\|<2$ \\ [0.5ex] 
 \hline
 4 & 5 & 2.38 & [0,0.438] \\ 
  \hline
 4 & 6 & 2.38 & [0,0.438] \\
  \hline
 4 & 7 & 2.38 & [0,0.438] \\ 
  \hline
 5 & 2 & 2.51 & [0,0.364] \\
  \hline
 5 & 3 & 2.51 & [0,0.368] \\
  \hline
 5 & 4 & 2.51 & [0,0.368] \\
  \hline
 6 & 2 & 2.63 & [0,0.306] \\ 
  \hline
 6 & 3 & 2.63 & [0,0.307] \\ 
  \hline
 6 & 4 & 2.63 & [0,0.307] \\ [0.5ex] 
  \hline
\end{tabular}
\caption{Summary of numerical investigations for $A_{t,m}$ and $X_{t,m}$ with different $m$ and $n$ values.}
\label{table:1}
\end{figure}
\end{remark}

We expect that similar arguments could be used to establish Crouzeix's conjecture for other classes of nilpotent matrices. However, because we often have $\| X_t \| \| X_t^{-1} \| >2,$ this line of argumentation will not be able to establish the entire conjecture, even for matrices of the form that we study in this paper.

\end{document}